\newcommand{\Mmb}{{\mathbb{M}}}
\newcommand{\NN}{{\mathbb{N}}}

\newcommand{\RR}{{\mathbb{R}}}
\newcommand{\EE}{{\mathbb{E}}}
\newcommand{\PP}{{\mathbb{P}}}

\newcommand{\clt}{{\mathcal{T}}}
\newcommand{\clp}{{\mathcal{P}}}

\newcommand{\cld}{{\mathcal{D}}}
\newcommand{\cle}{{\mathcal{E}}}
\newcommand{\clf}{{\mathcal{F}}}

\newcommand{\cls}{{\mathcal{S}}}

\newcommand{\Qbf}{{\mathbf{Q}}}

\newcommand{\Ubf}{{\mathbf{U}}}
\newcommand{\Vbf}{{\mathbf{V}}}

\newcommand{\Xbf}{{\mathbf{X}}}

\newcommand{\Zbf}{{\mathbf{Z}}}

\newcommand{\Om}{\Omega}
\newcommand{\om}{\omega}

\newcommand{\bfZ}{\mathbf{Z}}
\newcommand{\bfQ}{\mathbf{Q}}
\newcommand{\bfY}{\mathbf{Y}}
\newcommand{\bfM}{\mathbf{M}}
\newcommand{\bfU}{\mathbf{U}}
\newcommand{\bfB}{\mathbf{B}}

\newcommand{\bfv}{\mathbf{v}}
\newcommand{\bfz}{\mathbf{z}}
\newcommand{\bfx}{\mathbf{x}}
\newcommand{\bfy}{\mathbf{y}}
\newcommand{\bfL}{\mathbf{L}}

\newcommand{\bfa}{\mathbf{a}}
\newcommand{\id}{\mbox{\small{id}}}
\newcommand{\clr}{\mathscr{R}}

\newcommand{\diag}{\mbox{diag}}
\newcommand{\upp}{\upsilon}

\documentclass[10pt]{article}
\usepackage[portrait,margin=2.54cm]{geometry}

\usepackage{amsfonts}

\usepackage{amssymb}
\usepackage{mathrsfs}
\usepackage{soul}
\usepackage{hyperref}
\usepackage{amssymb,amsthm,amsmath,amsfonts,amsbsy,latexsym}
\usepackage{graphicx}
\usepackage{upref,setspace}
\usepackage{enumerate}
\usepackage{paralist}
\usepackage{color}

\usepackage{mathtools}

\definecolor{expcol}{rgb}{1.0,0.5,0.5}
\definecolor{ecol}{rgb}{0.0, 0.5, 1.0}

\numberwithin{equation}{section}
\numberwithin{figure}{section}
\numberwithin{table}{section}
\newtheorem{lemma}{Lemma}[section]
\newtheorem{proposition}[lemma]{Proposition}

\newtheorem{theorem}[lemma]{Theorem}

\newtheorem{corollary}[lemma]{Corollary}

\theoremstyle{remark}
\newtheorem{remark}[lemma]{Remark}

\numberwithin{equation}{section}

\mathtoolsset{showonlyrefs}

\newcommand{\cb}{\color{black}}

\begin{document}

\title{Load Balancing in Parallel Queues and Rank-based Diffusions}
\author{Sayan Banerjee, Amarjit Budhiraja, Benjamin Estevez}
\maketitle
\begin{abstract}
Consider a queuing system with $K$ parallel queues in which the server for each queue processes jobs at rate $n$ and the total arrival rate to the system is $nK-\upp \sqrt{n}$ where $\upp \in (0, \infty)$ and $n$ is large.
Interarrival and service times are taken to be independent and exponentially distributed. It is well known that the join-the-shortest-queue (JSQ) policy has many desirable load balancing properties. In particular, in comparison with uniformly at random routing, the time asymptotic {\cb total queue-length} of a JSQ system, in the heavy traffic limit, is reduced by a factor of $K$. However this decrease in {\cb total queue-length} comes at the price of a high communication cost of order  $nK^2$ since at each arrival instant, the state of the full $K$ dimensional system needs to be queried. In view of this it is of interest to study alternative routing policies that have lower communication costs and yet have similar load balancing properties as JSQ. 

In this work we study a family of such rank-based routing policies, {\cb which we will call \textit{Marginal Size Bias Load Balancing} (MSBLB) policies,} in which $O(\sqrt{n})$ of the incoming jobs
are routed to servers with probabilities depending on their ranked queue length and the remaining jobs are routed uniformly at random. A particular case of
such routing schemes, referred to as the marginal join-the-shortest-queue (MJSQ) policy, is one in which all the $O(\sqrt{n})$  jobs are routed using the JSQ policy. Our first result provides a heavy traffic approximation theorem for such  queuing systems in terms of  reflected diffusions in the positive orthant $\RR_+^K$. It turns out that, unlike the JSQ system where due to a state space collapse the heavy traffic limit is characterized by a one dimensional reflected Brownian motion, in the setting of MJSQ (and for the more general rank-based routing schemes) there is no state space collapse and one obtains a novel diffusion limit which is the constrained analogue of the well studied Atlas model (and other rank-based diffusions) that arise from certain problems in mathematical finance.
Next, we prove an interchange of limits ($t\to \infty$ and $n\to \infty$) result which shows that, under conditions, the steady state of the queuing system is well approximated by that of the limiting diffusion. It turns out that the latter steady state can be given explicitly in terms of product laws of Exponential random variables. Using these explicit formulae, and the interchange of limits result, we compute the time asymptotic {\cb total queue-length} in the heavy traffic limit for the MJSQ system. We find the striking result that although in going from JSQ to MJSQ the communication cost is reduced by a factor of $\sqrt{n}$, the steady state heavy traffic {\cb total queue-length} increases by at most a constant factor (independent of $n,K$) which can be made arbitrarily close to one by increasing a MJSQ parameter. We also study the case where the system is overloaded, namely $\upp<0$. For this case we show that although the $K$-dimensional MJSQ system is unstable, however, unlike the setting of random routing, the system has certain desirable and quantifiable load balancing properties. In particular, by establishing a suitable interchange of limits result, we show that   the steady state difference between the maximum and the minimum queue lengths stays bounded in probability (in the heavy traffic parameter $n$).


\noindent\newline


\end{abstract}

\section{Introduction}\label{intro}


Consider a system of $K$ parallel queues that are critically loaded. Specifically, denoting by $n \in \NN$ the heavy traffic parameter, the jobs arrive to the $i$-th queue according to a Poisson process with rate $n - c_i\sqrt{n}$ and these jobs are processed with Exponential service times with rate $n$, where $c_i \in [0,\infty)$ for $1 \le i \le K$. We make the usual assumption of mutual independence of all interarrival times and service times for the various queues in the system. Denoting by $Q^n_i(t)$ the number of jobs in the $i$-th queue (the state of the $i$-th queue) at time $t$, it follows from standard results that the $K$-dimensional process $\hat Q^n(t) \doteq (\hat Q^n_i(t)/\sqrt{n})_{i=1}^K$ converges in distribution (in the Skorokhod path space) to a $K$-dimensional normally reflected Brownian motion with drift vector $-\mathbf{c} = (-c_i)_{i=1}^K$, in the positive orthant $\RR_+^K$ \cite{ReimanOQN}. 

In recent years, with applications arising from large scale service centers, cloud computing platforms and data storage and retrieval systems \cite{LBcloud,altman2011load,gupta2007analysis,ongaro2011fast}, there has been a lot of interest in devising and studying properties of various types of load balancing schemes for parallel server networks. One of the basic forms of load balancing algorithms is the so-called join-the-shortest-queue (JSQ) policy in which jobs arrive to a central dispatcher and every incoming job is routed to the shortest queue in the system at that instant. JSQ has many desirable performance features; in particular, under many common assumptions on the distributions of the service times, such as when they are iid Exponential, the JSQ policy is `optimal' in the sense that it minimizes the expected time for a job to begin service once it has entered the queuing system 
 \cite{WinstonJSQ,WeberOpt}. Consequently, this policy has been extensively studied in many different directions \cite{graham2000chaoticity, graham2005functional, gupta2007analysis,bramson2011jsq, chen2012asymptotic, eschenfeldt2018join, BanMuk, brav, banerjee2020join,budhiraja2020jsq,mukherjee2022}.
 However, one challenging aspect of a JSQ policy is the high communication cost that is incurred in its implementation. Specifically, at each instant of arrival, the system manager needs to query the current state of all queues in the system, and since arrivals occur at rate $\approx nK$ (where $n$ is a large parameter) this operation can be expensive, particularly when $K$ is large as well. On the other hand, the parallel queuing system (PQS) of the form described in the first paragraph, which corresponds to routing incoming jobs at random (with routing probability to $i$-th queue proportional to $(1-c_in^{-1/2})$), is very easy to implement but does not have the desirable load balancing features of JSQ. For example, consider $K$ queues operating in parallel under the PQS policy with jobs arriving to the central dispatcher at rate $Kn - \upp\sqrt{n}$ {\cb and routed uniformly at random} (equivalently $c_i = v/K$ for $1\le i \le K$) where $\upp \in (0,\infty)$. The expected average steady state {\cb total queue-length} of the system for large $n$ is approximately $K\sqrt{n}\upp^{-1}$. In comparison, for the same central dispatcher system, with  the JSQ discipline, the expected average steady state {\cb total queue-length} is approximately $\sqrt{n}\upp^{-1}$ (see \cite[Table 1]{chen2012asymptotic}). Hence, the average {\cb total queue-length} (and by Little's law, the average time spent by a customer in the system) is reduced by a factor of $K$ in passing from the PQS policy to the JSQ policy.

The goal of the current work is to explore  load balancing schemes that in a certain fashion combine some of the desirable features of both types of policies discussed above. Specifically, let $a_i \in \mathbb{R}$, $1\le i \le K$ and let $\upp \doteq \sum_{i=1}^K a_i$ and $a_* \doteq \max_{1 \le i \le K}a_i$ {\cb and, for $1 \le i \le K$, write 
\begin{equation}\label{eq:rev1}
p_i =\frac{1 - a_i n^{-1/2}}{K - \upp n^{-1/2}} = \frac{nK - a_*K\sqrt{n}}{nK - v\sqrt{n}} \times \frac{1}{K} + \frac{(a_*K - v)\sqrt{n}}{nK - v\sqrt{n}}\times\frac{a_* - a_i}{a_*K-v}.
\end{equation}} 
For any $n \in \mathbb{N}$ with $n >  \max_{l \in \{1, \ldots, K\}} (a_l \vee 0)^2$, 
we consider a \emph{rank-based routing scheme} in which jobs arrive at a central dispatcher at rate $Kn - \upp\sqrt{n}$. Each job upon arrival is{\cb, with probability $p_j$,} routed to the $j$-th shortest queue in the system (with ties broken in the lexicographical order). {\cb We will call this family of routing schemes \textit{``Marginal Size Bias Load Balancing"} (MSBLB) policies. Intuitively, as clarified by the second equality in \eqref{eq:rev1},  an incoming job, upon arrival, is either (with probability $\frac{nK-a_*K\sqrt{n}}{nK-v\sqrt{n}}$) routed to a server uniformly at random or (with probability $\frac{(a_*K-v)\sqrt{n}}{nK-v\sqrt{n}}$) routed via a size-biased scheme, under which  the quantity $\frac{a_*-a_j}{Ka-v}$ can be thought of as the `preference' given to the $j$-th shortest queue. In particular, for each incoming job, the probability of random routing is $1 - O(n^{-1/2})$ and of routing via a size-biased scheme is $O(n^{-1/2})$.

} Of particular interest is the special case where $a_i =a$ for $2\le i \le K$ and $a_1 = a-b$, with $a,b>0$. In this case the system can be viewed as one in which arrivals occur to each queue (independently of others) at rate $n - a\sqrt{n}$, with the modification that the shortest queue gets an additional arrival stream of rate $b \sqrt{n}$. An advantageous  feature of this routing scheme, which we refer to as the \emph{marginal join-the-shortest-queue (MJSQ)} policy, is that the communication cost for this policy is of 
 the order $\sqrt{n} Kb$ rather than of the order $nK^2$ for the JSQ policy. One expects that there is a performance loss in terms of load balancing when using the MJSQ instead of the JSQ policy, which in some way offsets the  advantage in terms of the communication cost involved in implementing the policy. Our objective in this work is to systematically study rank-based load balancing schemes of the above form, in the heavy traffic regime, with an eye towards examining such tradeoffs in a precise and rigorous manner.
 
 We are in particular interested in establishing limit theorems for such systems that characterize the heavy traffic asymptotics of diffusion scaled queue-length processes in terms of suitable constrained rank-based diffusion processes
 and in studying the long time behavior of the queuing systems and of their diffusion limits. Towards the latter goal, we establish, under conditions, an interchange of limits ($t \to \infty$, $n\to \infty$) result that allows the approximation of the steady state behavior of the diffusion scaled queuing system by that of its diffusion limit. Furthermore, we show that the stationary distribution of the limiting diffusion can be given explicitly in terms of certain product laws of Exponential random variables. These exact formulae allow us to compute heavy traffic approximations of key steady state performance metrics for various rank-based load balancing schemes of the form discussed above. In particular, we obtain the following striking results for the MJSQ routing policy (see Section \ref{mjsqrem} for details).
  %
 %
 %
 %
 %
 %
 When the net arrival rate is $nK-\upp\sqrt{n}$ and each server has service rate $n$, the expected, per-server, steady state queue-length, scaled by $\sqrt{n}$, for large $n$, is approximately $K/\upp$ for PQS, $1/\upp$ for JSQ and $(K-1)/(Ka) + 1/\upp$ for MJSQ (with $a >0$ and $b=aK-\upp$ for some $\upp \in (0, aK)$). Thus,  the MJSQ policy significantly improves the performance over PQS and brings the average queue-length down to a constant factor $\approx (1 + \upp/a)$ of JSQ. 
Moreover, the system imbalance index, defined as the expected  (steady state) difference between the maximum and minimum queue lengths, is reduced by a factor $Ka/\upp$ in going from the PQS to the MJSQ policy.  Thus, when $a\to \infty$, the system approaches a form of state space collapse in the steady state. For the heavy traffic limit of the JSQ system, an exact process level state space collapse is known from the work of \cite{chen2012asymptotic}. {\cb The phenomenon of state space collapse in queueing systems in heavy traffic is studied more generally in \cite{SubdifLB}.} Thus our results show that, for fixed $\upp$, the MJSQ policy becomes `asymptotically optimal' as $a \rightarrow \infty$ in the sense that the steady state average queue-lengths and the system imbalance index approach the analogous quantities for JSQ.
 Further, as noted in Remark \ref{unload},  even for the range of parameters where both the PQS and MJSQ are unstable (i.e. $\upp <0$), the MJSQ policy has certain desirable and quantifiable load balancing properties.
 Specifically, by establishing a suitable interchange of limits result, our results  show that the steady state difference between the maximum and the minimum queue lengths stays bounded in probability (in the heavy traffic parameter $n$) and in fact it vanishes to $0$ as the system overload parameter becomes large (i.e. $\upp \to -\infty$).

 We now describe our main results in more detail. In Theorem \ref{convtodiff} we provide our heavy traffic limit theorem which gives an approximation for the diffusion scaled queue length process, on compact time intervals, in a suitable path space, in terms of a certain reflected diffusion process in the nonnegative orthant $\RR_+^K$. We find that, in the study of this asymptotic behavior, it is more convenient to consider the evolution equation for the ranked queues rather than for the original labeled queues. Theorem \ref{convtodiff} provides the asymptotic behavior of the shortest queue, together with the gaps between all the successively ranked queues, from which the asymptotic behavior of the ranked queuing system follows immediately. As discussed in Remark  \ref{rem:atl}, the limiting diffusion can be viewed as the constrained version of certain rank-based diffusions (e.g. the Atlas model) that in recent years have been studied extensively, motivated by certain problems in mathematical finance. This connection is discussed further in the next paragraph. We also observe that the situation here is quite different from the setting of a JSQ system where one finds \cite{chen2012asymptotic} that, due to a certain state space collapse property, the limiting diffusion can be characterized in terms of a one dimensional reflected Brownian motion. In contrast, here there is no state space collapse and a novel reflected diffusion emerges as the heavy traffic limit of the queuing system. We are particularly interested in the long time behavior. In Theorem \ref{statdistthm}, using results of \cite{harwil},  we identify a simple necessary and sufficient condition for the positive recurrence of the limiting diffusion. Furthermore, we show using \cite{harwil2}  that when this condition is satisfied the unique stationary distribution of the reflected diffusion has an explicit product form Exponential law. Next, in Theorem \ref{statdistconv} we show that, under the same stability condition, for large 
 $n$, there is a unique stationary distribution for the shortest queue together with the gap sequence associated with the ranked queue-length processes, and, under the diffusion scaling, these stationary distributions converge to the unique stationary distribution for the limiting diffusion. Finally, in Theorem \ref{unstabstatdist} we investigate a setting where the system is overloaded and consequently both the MJSQ queuing processes and the PQS processes are unstable. The theorem shows that even in this unstable setting MJSQ policy has certain desirable load balancing properties. Specifically, the queue-length and gap processes for the PQS are transient while for the MJSQ system, although the queue-length processes are transient, the gap processes are stable in a suitable sense. We note that the gap processes by themselves (not including the shortest queue process) are not Markovian, nevertheless our results allow us to study its steady state behavior by
 showing that, for sufficiently large $n$,  the laws of the gaps at time instant $t$ converge as $t \to \infty$. Furthermore, the limiting laws converge, as $n \to \infty$, to an Exponential product form law. This latter distribution can in fact be identified as exactly the stationary distribution of the gap process for a $K$ particle standard (unconstrained) Atlas model (see also next paragraph). Using this explicit expression for the stationary distribution we then identify a useful load balancing property of the MJSQ policy in this unstable regime that says that,  unlike the PQS for which the difference between the maximum and minimum queue lengths is not tight over time, for the MJSQ system this difference converges in distribution to a finite random variable as $t \to \infty$ and $n \to \infty$. The expectation of this random variable, which is approximately of order $Kb^{-1} \log K$ for large $K$, gives a precise quantitative measure of load balancing achieved by MJSQ over PQS (for which this quantity is $\infty$).


As alluded to in the above discussion, through its heavy traffic limiting behavior, the rank-based routing policies discussed above are connected to another area of much recent interest  - namely,  the study of \textit{rank-based diffusions} \cite{Atlas2,Atlas1,Atlas3,DJO,sarantsev2017stationary,AS,banerjee2021domains}. These diffusions are models that involve a collection of (Brownian) particles on the real line whose drift and diffusivity vary over time according to the relative ranks of their positions. These models have a variety of intriguing features, one of which is that they frequently exhibit product-form stationary distributions. A basic example of such a rank-based diffusion is the well-known \textit{Atlas model} \cite{Atlas1,Atlas2,Atlas3,IchKar,IchHybrid} in $K$ particles, introduced by Fernholz in the context of stochastic portfolio theory \cite{Fernholz}, in which the particle of lowest rank at any time evolves as a Brownian motion with strictly positive drift, while the other $K-1$ particles evolve as standard Brownian motions. As is noted in Remark \ref{rem:atl}, the diffusion processes arising in the heavy traffic limit of our rank-based routing policies can be viewed as
 the gaps between ordered particles in certain constrained rank-based diffusion models. 
Specifically, in the setting of the MJSQ policy,
 the diffusion which arises in the heavy traffic limit for the gaps between ranked queues  is a variation of the Atlas model in which the lowest particle is reflected at zero. 
 The connection with the Atlas model becomes even more direct in the setting where  the MJSQ policy is unstable. In this case the lowest particle escapes the origin in finite time, leading to a limiting stationary dynamics for the upper $K-1$ ranked gaps which is identical to that of the Atlas model with $K$ particles. However there is an important distinction between the two models in that for the Atlas model the gap processes describe a Markov process and although the lowest particle state process is not bounded in probability, the gap processes are positive recurrent. In contrast, for the limiting diffusion arising from the (stable) MJSQ system (and from other stable rank-based policies), due to the reflection at the boundary, the gap process by itself is not Markovian, however the lowest particle state process together with the gap processes describe a positive recurrent Markov process.

A class of load balancing policies that have attracted much attention in recent years \cite{mitzenmacher2001power, mitz2, vdk, blp,msy,
mukherjee2018universality,budhiraja2019diffusion,luczak2006maximum} are the so called 
\emph{Power-of-choice (PoC)} schemes. In such a scheme, upon each arrival $d \in \{2,\dots,K\}$ servers are chosen uniformly at random and the job is assigned to the shortest of the $d$ queues. The attractiveness of these policies is in their low communication cost which is of order $ndK$ as opposed to the order $nK^2$ cost associated with the JSQ policy. When $K$ is large and $d$ is small (e.g. $d=2$) the advantage can be significant. From \cite{chen2012asymptotic}  it is known that the heavy traffic limit for a PoC scheme for a system of the form considered in this work is the same for any $d= 2, 3, \ldots, K$ and thus, in particular, is characterized in terms of a one dimensional reflected Brownian motion. Although the PoC achieves the same  asymptotically optimal behavior as JSQ (in the setting of a fixed $K$ number of queues) while incurring a lower communication cost, MJSQ still has significant advantage over the PoC in terms of these costs ($\sqrt{n}bK$ versus $ndK$) for large $n$ (and fixed $K$). However when $K$ is large (growing with $n$) this comparison becomes less clear and it would be interesting to explore the performance of the `marginal' version of the PoC scheme in which only $O(\sqrt{n})$ many jobs are routed using a PoC scheme and the rest are routed uniformly at random. Even for fixed $K$, unlike the setting of \cite{chen2012asymptotic} where the diffusion limit is the same for a JSQ scheme and a PoC scheme, we expect the diffusion limit for the marginal PoC scheme under the heavy traffic scaling to be different from the diffusion obtained in Theorem \ref{convtodiff} for the MJSQ scheme. When both $n$ and $K$ are suitably large, the differences between the JSQ and PoC are particularly striking \cite{mitzenmacher2001power,mukherjee2018universality,VDBsurv} and it would be of interest to explore the performance of MJSQ and marginal version of PoC schemes in this asymptotic regime as well. 

{\cb Another 
well studied load balancing scheme is the so called 
 ``Join the Idle Queue" (JIQ) policy, described in \cite{JoinIdle}. Under this policy, rather than the dispatcher querying the full state of the system to determine which server to route incoming jobs to, any server whose queue empties instead notifies the dispatcher that it has fallen idle, and the next incoming job is routed to this server, or, if no server has been flagged as idle, the incoming job is routed uniformly-at-random. In \cite{JoinIdle}, numerical evidence is given that JIQ outperforms a Power-of-2 scheme with respect to load-balancing efficiency. 
 Other examples of load balancing schemes are the ``Persistent Idle" policy, introduced in \cite{PersIdle} and
 ``Load balancing with memory" algorithm introduced in \cite{MemLB}. For the first scheme, 
 incoming jobs are routed to the server which last reported idle to the dispatcher until another server (different from the previous server) reports as idle to the dispatcher. This policy has been shown numerically to have comparable load-balancing performance to JSQ under certain conditions in \cite{PersIdle}. 
 The second scheme works as follows. At initialization, the server has a record of the lengths of the queues at each server. When a job arrives, the dispatcher queries the queue-lengths of a random sample of (fixed) size $d < K$ of the servers, updates its record of queue-lengths for the sampled servers only, and then dispatches the incoming job to the server which has the shortest queue according to its record. It would be of interest to study performance of `marginal' versions of these policies
 in which only $O(\sqrt{n})$ of the incoming jobs per unit time use such a scheme while the remaining are routed uniformly at random.} 
 
Finally, in this work we have assumed the network primitives to be Exponentially distributed which simplifies many arguments. It would be interesting to explore the setting where the interarrival and/or service times are not Exponentially distributed.
These questions are left for future study.

\subsection{Setting and Notation}
\label{sec:notat}
For $j \in \mathbb{N}$, we will write $[j] \doteq \{1,...,j\}$. For a $d\times d$ matrix $A$, $\diag(A)$ will denote the $d\times d$ diagonal matrix with diagonal entries given by the diagonal entries of $A$. We denote by $I$ the identity matrix whose dimension will be clear from the context.
All inequalities involving vector quantities are to be interpreted componentwise. 
For a $K$-dimensional vector $\bfx= (x_1, \ldots x_K)'$ in $\RR^K$ and $j \in [K]$ we  define
$r_j(\bfx) = i$ if the $i$-th coordinate of $\bfx$ has rank $j$, with ties broken in the lexicographical order.
We will sometimes write $\id: [0,\infty) \rightarrow \mathbb{R}$ for the function $\id(t) = t$, $t \geq 0$. 
For a metric space $\mathbb{X}$, we denote by $\mathcal{D}([0,\infty):\mathbb{X})$ the space of functions $f: [0,\infty) \rightarrow \mathbb{X}$ which are right-continuous and have finite left-limits (RCLL) endowed with the usual Skorokhod topology and by $\mathcal{D}_0([0,\infty):\mathbb{R}^K)$ the set of all $f \in \mathcal{D}([0,\infty):\mathbb{R}^K)$ such that $f(0) \geq 0$. 
The space of continuous functions $f: [0,\infty) \rightarrow \mathbb{X}$ equipped with the topology of local uniform convergence will be denoted as $\mathcal{C}([0,\infty):\mathbb{X})$.
For a Polish space $\cls$, $\clp(\cls)$ will denote the space of probability measures on $\cls$ equipped with the topology of weak convergence.
We denote convergence in distribution of random variables with values in some metric space and defined on some probability space by $\Rightarrow$. With an abuse of notation, we also write $\nu_n \Rightarrow \nu$ for weak convergence of a sequence of probability measures $\{\nu_n\}$ to some probability measure $\nu$. 
For $j \in \mathbb{N}$, $f \in \mathcal{D}([0,\infty):\mathbb{R}^j)$ and $T \geq 0$, we will write the supremum norm over the interval $[0,T]$ as $||f||_{j,T} \doteq \sup_{s \in [0,T]}|f(s)|_j$, where $|\cdot|_j$ is the standard Euclidean norm in $\mathbb{R}^j$. The subscript $j$ will be omitted where it may be inferred from context.

Throughout this work, multiple distinct but distributionally equivalent  representations of the same queuing processes will be used for the sake of making efficient arguments. In order to preserve clarity, the symbols $A, A_1, A_2, ...$ and $D, D_1, D_2, ...$ will be reused throughout to refer to  Poisson processes defined on some probability space $(\Omega,\mathscr{F},\mathbb{P})$. These symbols should be understood to refer to the same processes within any individual proof, but their definitions may differ between proofs as needed.

\section{Model description}\label{moddesc}

Let $n, K \in \mathbb{N}$, and $a_i \in \mathbb{R}$ for $i = 1, ..., K$. {\cb Throughout we assume that $n$ is large enough so that $n - \max_{i \in [K]}a_i\sqrt{n} > 0$, namely we assume that $n >  n_* \doteq \max_{l \in [K]} (a_l \vee 0)^2$.} Consider $K$ queues labeled $1, ..., K$. For $i \in [K]$, a stream of jobs arrives at the $i^{\textnormal{th}}$ shortest of these queues in accordance with a Poisson process with rate $n - a_i\sqrt{n}$. 
The jobs are served at the queue at which they arrive according to the first-come-first-serve (FCFS) policy, and service times are given by  Exp($n$) random variables. We assume mutual independence of all service times, interarrival times, and routing indicator random variables that govern the streaming of jobs to the various queues in the system. Specifically consider a central {\cb dispatcher} to which jobs arrive according to a Poisson process $\Lambda^n$ with rate $nK - \upp\sqrt{n}$, where as before, $\upp = \sum_{i=1}^K a_i$. 
{\cb We consider a routing policy under which each job upon arrival is routed to the $j$-th shortest queue in the system
with probability $p_j$, defined as in \eqref{eq:rev1}
 (with ties broken in the lexicographical order). Mathematically this can be described as follows. }
Let  $\{U_m\}_{m\in \NN}$ be an iid sequence of random variables with values in $\{1, \ldots , K\}$ where
$$P(U_m=j) = \frac{1 - a_j n^{-1/2}}{K - \upp n^{-1/2}}, \; j \in [K], \; m \in \NN.$$
Since $n >  n_*$, the above is indeed a well-defined probability measure.
Let $\{v_m\}_{m \in \NN}$ be an iid sequence of Exp($n)$ random variables. We assume that the Poisson process $\Lambda^n$, the sequence $\{U_m\}_{m \in \NN}$ and the sequence
$\{v_m\}_{m\in \NN}$ are mutually independent. With this notation, the above queuing system corresponds to the setting where the $m$-th arriving job according to the Poisson process $\Lambda^n$ is routed to the $U_m$-th shortest queue at the instant of arrival (with ties broken according to the lexicographical order) and its processing time (from the time the service starts) is $v_m$. {\cb Thus, as seen from \eqref{eq:rev1}, each job upon arrival is, with probability $\frac{nK-a_*K\sqrt{n}}{nK-v\sqrt{n}}$, routed uniformly at random to one of the $K$-queues and, for $1\le j\le K$, routed to the $j$-th shortest queue in the system with probability 
$\frac{\sqrt{n}(a_* - a_i)}{nK - v\sqrt{n}}$ (recall $a_* \doteq \max_{1 \le i \le K}a_i$).}

Denote by $Q^n_i(t)$ the length of the $i^{\textnormal{th}}$ queue at time $t \geq 0$ and by $X^n_i(t)$ the length of the $i^{\textnormal{th}}$ shortest queue at time $t$ in the $n$-th system (in particular, $X^n_1$ gives the length of the shortest queue). We will suppress $n$ from the notation unless needed. When multiple queues are tied in length, we uniquely identify the unordered queues with the ordered queues in accordance with the lexicographic ordering of their indices. That is, if at a time $t \geq 0$ and for $\ell \in [K]$, $i_1 < i_2 < ... < i_\ell \in [K]$ are the indices of the unordered queues tied for $j$-th shortest queue, then $Q_{i_1}(t) = X_j(t), Q_{i_2}(t) = X_{j+1}(t), ..., Q_{i_{\ell}}(t) = X_{j + \ell - 1}(t)$. Note that in the above, when we refer to `the $i^{\textnormal{th}}$ shortest queue', we mean in particular the queue whose rank is $i$ under the aforementioned tie-breaking scheme.

Write $\Xbf(\cdot) \doteq (X_1(\cdot),...,X_K(\cdot))'$.
 We will now present a convenient state evolution equation for this system in terms of a collection of independent Poisson processes. Such time-changed Poisson process representations have been used extensively in the literature and originate from the work of Kurtz \cite{Kurtz80}. Before we present the state equation we give a heuristic description of the representation.

We work on a probability space $(\Omega, \mathscr{F}, \mathbb{P})$ on which are defined $2K$ mutually independent rate $1$ Poisson processes $A_1(\cdot),...,A_K(\cdot)$ and $D_1(\cdot),...,D_K(\cdot)$. For $j \in [K]$, the process $A_j$ will be used to define the arrival process for the $j$-th ranked queue, and $D_j$ will be used for defining the departure process for the same queue. Note that the actual identity (i.e. the label) of the $j$-th ranked queue changes dynamically over time.

We consider the $i$-th ranked queue, for $i \in [K-1]$. Suppose that at a time $t \geq 0$, $X_i(t-) = X_{i+1}(t-)$. Then at time $t$, $X_i$ cannot increase because otherwise it would pass above $X_{i+1}$ and therefore become the $i+1$-th ranked queue. Hence, in this circumstance, the instantaneous rate of arrival at $X_i$ at time $t$ must be zero. Similarly, if we instead suppose that for $i\in [K]$, $X_i(t-) = X_{i-1}(t-)$ (we set $X_0(t) =0$ for all $t\ge 0$), then we have that $X_i$ cannot decrease, because otherwise it would pass below $X_{i-1}$, and therefore the instantaneous rate of departure at $X_i$ at time $t$ must be zero.

Now suppose instead that, for $i \in [K]$ (by convention we set $X_{K+1}(t) =\infty$ for all $t\ge 0$), $X_i(t-) < X_{i+1}(t-)$. Let $J \subset [i-1]$ be the (possibly empty) subset of indices such that $X_j(t-) = X_i(t-)$ when $j \in J$ and $X_j(t-) < X_i(t-)$ when $j \in [i-1]\setminus J$. Since for $j \in J$, the queue $X_j$ is lower-ranked than $X_i$, any arrival which occurs at the (unranked) queue corresponding to $X_j$ causes $X_i$ to increase. Hence, in this case, the instantaneous rate of arrival at $X_i$ is 
\begin{equation*}
    (n - a_i\sqrt{n}) + \sum_{\{j \in [i-1]: X_j(t-) = X_i(t-)\}}(n - a_j\sqrt{n}).
\end{equation*} 
Now suppose instead that, for $i \in [K]$, $X_i(t-) > X_{i-1}(t-)$, and let $J \subset \{i+1,...,K\}$ 
be the subset of indices such that when $j \in J$, $X_i(t-) = X_j(t-)$ and when $j \in \{i+1,...,K\}\setminus J$, $X_i(t-) < X_j(t-)$. Since for $j \in J$, the queue $X_i$ is lower-ranked than $X_j$, any departures at the (unordered) queue corresponding to $X_j$ will lead to a decrease in $X_i$. Hence, in this case, the instantaneous rate of departure at $X_i$ is 

\begin{equation*}
    n + n\#\{j \in \{i+1,...,K\}: X_i(t-) = X_j(t-)\},
\end{equation*}
where for a finite set $S$, $\#S$ denotes its cardinality.

Let, for $n \in \NN$, $\{X_i^n(0)\}_{i \in [K]}$ be a collection of $\NN_0$ valued random variables that are independent of 
$\{A_i, D_i, 1 \le i \le K\}$ such that $X^n_1(0) \le X^n_2(0) \le \cdots \le X^n_K(0)$. These will represent the initial states of the ordered queues.
Then with these considerations, we  can now describe the evolution equations for $X^n_1, X^n_2, ..., X^n_{K}$ in terms of the Poisson processes $\{A_i, D_i, 1 \le i \le K\}$ as follows
\begin{align}
    X^n_i(t) = X^n_i(0) + & A_i\left(\int_0^t (n - a_i\sqrt{n} + \sum_{j = 1}^{i-1}(n - a_j\sqrt{n})\mathbf{1}_{\{X^n_i(s) = X^n_j(s))\}} )\mathbf{1}_{\{X^n_i(s) < X^n_{i+1}(s)\}} )ds\right) \\
    - & D_i\left(\int_0^t (n + n\sum_{j = i+1}^{K}\mathbf{1}_{\{X^n_i(s) = X^n_j(s)\}})\mathbf{1}_{\{X^n_{i}(s) > X^n_{i-1}(s)\}} ) ds\right), \; i \in [K].
\end{align}
Note in particular that for the $K$-th queue, there is no upper queue which restricts the arrival process and so the indicator
$\mathbf{1}_{\{X^n_K(s) < X^n_{K+1}(s)\}}$ in the arrival process term can be dropped and
we have the following expression for departure term
\begin{align}
    D_K \left(\int_0^t (n + n\sum_{j = K+1}^{K}\mathbf{1}_{\{X^n_K(s) = X^n_j(s)\}})\mathbf{1}_{\{X^n_{K}(s) > X^n_{K-1}(s)\}} ) ds\right)= D_K\left(n\int_0^t\mathbf{1}_{\{X^n_K(s) > X^n_{K-1}(s)\}}ds\right).
\end{align}
For $i \in K$, write 

\begin{align*}
    \delta^n_i(t) \doteq \sum_{j=1}^K\mathbf{1}_{\{X^n_j(t) = X^n_i(t)\}}, \;\; 
    \alpha_i^n(t)  \doteq \sum_{j=1}^K(1-{a}_jn^{-1/2})\mathbf{1}_{\{X^n_j(t) = X^n_i(t)\}}.
\end{align*}
{\cb Note that the following identities hold
\begin{multline*}
n - a_i\sqrt{n} + \sum_{j = 1}^{i-1}(n - a_j\sqrt{n})\mathbf{1}_{\{X^n_i(s) = X^n_j(s)\}} = 
\sum_{j = 1}^{i}(n - a_j\sqrt{n})\mathbf{1}_{\{X^n_i(s) = X^n_j(s)\}}\\ 
= \left(\sum_{j = 1}^{K}(n - a_j\sqrt{n})\mathbf{1}_{\{X^n_i(s) = X^n_j(s)\}}\right)\mathbf{1}_{\{X^n_i(s) < X^n_{i+1}(s)\}}
=n\alpha_i^n(s)\mathbf{1}_{\{X^n_i(s) < X^n_{i+1}(s)\}}.
\end{multline*}
Similarly,
\begin{multline*}
n\left(1+ \sum_{j = i+1}^{K}\mathbf{1}_{\{X^n_i(s) = X^n_j(s)\}}\right) = n\sum_{j = i}^{K}\mathbf{1}_{\{X^n_i(s) = X^n_j(s)\}}\\ 
= n\left(\sum_{j = 1}^{K}\mathbf{1}_{\{X^n_i(s) = X^n_j(s)\}}\right)\mathbf{1}_{\{X^n_{i}(s) > X^n_{i-1}(s)\}} = 
n\delta^n_i(s)\mathbf{1}_{\{X^n_{i}(s) > X^n_{i-1}(s)\}}.
\end{multline*}
and therefore, recalling} the convention $X^n_0(\cdot) \doteq 0$ and $X^n_{K+1}(\cdot) \doteq \infty$, we can  rewrite the state evolution as
\begin{align}
    X^n_i(t) =  X^n_i(0) + A_i(n\int_0^t \alpha^n_i(s)\mathbf{1}_{\{X^n_i(s) < X^n_{i+1}(s)\}}ds) - D_i(n\int_0^t \delta^n_i(s)\mathbf{1}_{\{X^n_{i}(s) > X^n_{i-1}(s)\}}ds),\,\,\,\,i \in [K].\nonumber\\ \label{ordqs}
\end{align}
Note that the existence and uniqueness of strong solutions to \eqref{ordqs}  follows by a standard jump-by-jump iterative construction.

\section{Main Results}\label{results}


Define the $K \times K$ matrix $\clr$ as
\begin{align*}
	\clr_{ii} &= 1, \, i \in [K], \clr_{i, i+1} = -1/2, \, i \in [K-1], \; \clr_{i, i-1} = -1/2, 3 \le i \le K, \clr_{2, 1} = -1, \; \clr_{i, j} = 0 \mbox{ for other } i,j.
\end{align*}
Namely,
\begin{equation}\label{XRmat}
\mathscr{R} \doteq
\begin{pmatrix}
 1 & -\frac{1}{2} & 0 & 0 & \cdots & 0 \\
 -1 & 1 & -\frac{1}{2} & 0 & \cdots & 0 \\
 0 & -\frac{1}{2} & 1 & -\frac{1}{2} & \cdots & 0 \\
 \vdots & \vdots & \vdots & \vdots & \vdots & \vdots \\
 0 & \cdots & \cdots & \cdots & -\frac{1}{2} & 1
\end{pmatrix}.
\end{equation}


In subsequent sections, we will often wish to rewrite an  evolution equation for a stochastic dynamical system in terms of a Skorokhod map with respect to some reflection matrix applied to an input path. 
Denote by $\Mmb$ the collection of all $K\times K$ matrices $M$ that can be written in the form $M=(I-Q')$  where  $Q$ is a nonnegative matrix with zeroes on the diagonal and spectral radius strictly less than $1$.
It is easy to verify that the matrices  $\clr$ and the identity matrix $I$  both belong to the class $\Mmb$. It follows from the results of \cite{harrei} that the Skorokhod problem associated with any matrix $M \in \Mmb$ is wellposed. More precisely, the following result holds. 
%
%
%
\begin{proposition}\label{prop:skor}
	Fix $M \in \Mmb$.
Then  for every $x \in \mathcal{D}_0([0,\infty): \mathbb{R}^K)$, there exists a unique pair $(\eta,y) \in \mathcal{D}([0,\infty): \mathbb{R}_+^K) \times \mathcal{D}([0,\infty): \mathbb{R}_+^K)$ such that,
\begin{itemize}
    \item[(i)] for all $t \geq 0, y(t) = x(t) + M\eta(t)$,
    \item[(ii)] for each $i \in [K]$, (a) $\eta_i(0) = 0$, (b) $t\mapsto \eta_i(t)$ is non-decreasing, (c) $\int_0^{\infty}y_i(t)d\eta_i(t) = 0$.
\end{itemize}

The pair $(\eta, y)$ is called the solution to the Skorokhod problem for $x$ with respect to $M$. The map $\Gamma: \mathcal{D}_0([0,\infty): \mathbb{R}^K) \rightarrow \mathcal{D}([0,\infty): \mathbb{R}_+^K) \times \mathcal{D}([0,\infty): \mathbb{R}_+^K)$ given by 
\begin{equation*}
    \Gamma(x) = (\eta,y) = (\Gamma_1(x),\Gamma_2(x))
\end{equation*}
is Lipschitz-continuous in the sense that there exists $c_{\Gamma} \in (0,\infty)$ such that for $x, x' \in \mathcal{D}_0([0,\infty): \mathbb{R}_+^K)$ and $t < \infty$,
\begin{equation*}
    ||\Gamma_1(x) - \Gamma_1(x')||_t + ||\Gamma_2(x) - \Gamma_2(x')||_t \leq c_\Gamma||x - x'||_t.
\end{equation*}
\end{proposition}

{\cb In Section \ref{theorres}, we present our main diffusion approximation result and study long time stability of the  general model described in Section \ref{moddesc}. In Section \ref{mjsqrem}, we 
make several observations for the special case of the MJSQ policy introduced in Section \ref{intro} and present several quantitative comparisons of its load-balancing properties to those of the PQS and JSQ policies. In Section \ref{unstabsec}
we consider a setting of an overloaded system and provide results demonstrating good load balancing properties of the MJSQ policy in such a regime as well.
 In Section \ref{organization}, we provide an overview of the structure of the remainder of the paper.
}
\subsection{Diffusion Limit and Stability}\label{theorres}

For $t \geq 0$, let $\mathbf{X}^n(t) \doteq (X^n_1(t), \ldots, X^n_K(t))'$ and define the diffusion-scaled ranked queue-length and gap processes
\begin{align*}
    \hat{\mathbf{X}}^n(t) \doteq \frac{\mathbf{X}^n(t)}{\sqrt{n}},\,\,\,\,\,\hat{\mathbf{Z}}^n(t) \doteq (\hat{X}^n_1(t),\hat{X}^n_2(t)-\hat{X}^n_1(t),...,\hat{X}^n_K(t)-\hat{X}^n_{K-1}(t))'.
\end{align*}
Note that both $ \hat{\mathbf{X}}^n(t)$ and $\hat{\mathbf{Z}}^n(t)$ take values in $\RR_+^K$.

Our first result gives a heavy traffic limit theorem providing convergence in distribution of the diffusion-scaled gap process $\hat{\mathbf{Z}}^n(\cdot)$ to a reflecting Brownian motion in the positive orthant. {\cb Note that 
a diffusion limit for $\hat{\Xbf}^n$ may be recovered immediately from that for $\hat{\mathbf{Z}}^n$ via a linear transformation. However, technically it is somewhat simpler to work with $\hat{\mathbf{Z}}^n$ rather than with 
$\hat{\Xbf}^n$ as the former leads to a diffusion in the positive orthant for which the classical results of Harrison and Reiman \cite{harrei} and Harrison and Williams \cite{harwil,harwil2} can be readily applied, whereas for the latter one needs to consider diffusions in the Weyl chamber $\{\mathbf{x}: 0 \le x_1 \le x_2 \ldots \le x_k\}$.}

\begin{theorem}\label{convtodiff}
    Suppose that 
	$\hat{\Zbf}^n(0)$ converges in distribution to some limit $\Zbf(0) = (Z_1(0), \ldots, Z_K(0))'$. Then, as $n \to \infty$,
   $\hat{\Zbf}^n(\cdot) \Rightarrow\,\, \Zbf(\cdot)$ in $\mathcal{D}([0,\infty): \mathbb{R}_+^K)$, where $\Zbf = (Z_1,\dots, Z_K)'$ is a continuous process given as the solution of the following system of equations:
    For every $t \ge 0$,
\begin{align}
\begin{split}\label{z0eq}
    Z_1(t) & = Z_1(0) + \sqrt{2}B_1(t) - a_1t - \frac{1}{2}L_2(t) + L_1(t), \\
    Z_2(t) & = Z_2(0) + \sqrt{2}(B_2(t) - B_1(t)) - (a_2 - a_1)t -\frac{1}{2}L_3(t) + L_2(t) - L_1(t), \\
    Z_i(t) & = Z_i(0) + \sqrt{2}(B_i(t) - B_{i-1}(t)) - (a_i - a_{i-1})t -\frac{1}{2}(L_{i+1}(t) + L_{i-1}(t)) + L_{i}(t),
\end{split}
\end{align}
for $i \in \{3,\ldots,K\}$, where $\{B_i\}_{i \in [K]}$ is a collection of mutually independent standard real Brownian motions, independent of $\Zbf(0)$, and $L_i(\cdot)$ is the  local time of $Z_i(\cdot)$ at $0$, namely a continuous nondecreasing process starting at $0$, and  satisfying $L_i(t) = \int_0^t\mathbf{1}_{\{Z_i(s) = 0\}}dL_i(s)$, $i\in [K]$, and (by convention) $L_{K+1}(t) \doteq 0$ for all $t \geq 0$.
\end{theorem}
We can compactly rewrite \eqref{z0eq} in matrix form  as follows
\begin{equation*}
    \mathbf{Z}(t)  = \mathbf{Z}(0) + \boldsymbol{\rho}t + A\mathbf{B}(t) + \mathscr{R}\mathbf{L}(t),
\end{equation*}
where $\boldsymbol{\rho} = (-a_1, -(a_2 - a_1), \ldots, -(a_K-a_{K-1}))'$, $\mathbf{B}(\cdot) = (B_1(\cdot),...,B_K(\cdot))'$, $\mathbf{L}(t) \doteq (L_1(t),...,L_K(t))'$ and
\begin{equation}\label{Amat}
A  \doteq
\begin{pmatrix}
\sqrt{2} & 0 & 0 & \cdots & 0 & 0 \\
-\sqrt{2} & \sqrt{2} & 0 & \cdots & 0 & 0 \\
0 & -\sqrt{2} & \sqrt{2} & \cdots & \vdots & \vdots \\
\vdots & \vdots & \vdots & \ddots & \sqrt{2} & 0 \\
0 & 0 & 0 & \cdots & -\sqrt{2} & \sqrt{2} \\
\end{pmatrix}
\end{equation}
is a $K \times K$ matrix.
 Using the Skorhod map formalism of Proposition \ref{prop:skor} we see that the pair $(\mathbf{Z}, \mathbf{L})$ is uniquely characterized as
$$(\mathbf{Z}, \mathbf{L}) = \Gamma(\mathbf{Z}(0) + \boldsymbol{\rho}\,\id+ A\mathbf{B}),$$
where $\Gamma$ is the Skorokhod map associated with the reflection matrix $\clr$.
\begin{remark}\label{rem:atl}
	Let $\bfZ(0)$ and $\{B_i\}_{i \in [K]}$ be as in Theorem \ref{convtodiff}. Define a $\RR_+^K$ valued random variable $\bfY(0) \doteq (Z_1(0), Z_1(0) + Z_2(0), \ldots , Z_1(0)+ \ldots + Z_K(0))'$. Suppose that for some $a,b \in \RR$,
	$a_1 = a-b$ and $a_i=a$ for $2\le i \le K$.
	For $\mathbf{y} \in \RR_+^K$, 
    let $e(\mathbf{y}) \doteq e_{r_1(\mathbf{y})}$, where $\{e_j\}_{j \in [K]}$ is the canonical basis in $\RR^K$.
	Denote by $\bfa \in \RR^K$ the vector
	$(a, \ldots, a)'$.
	Then using Girsanov's theorem it is easily seen that the following equation has a unique weak solution
	\begin{equation} \label{eq:atwref}(\bfY, \bfL_Y) = \Gamma^N\left(\bfY(0) -\bfa\, \id + b \int_0^{\cdot} e(\bfY(s)) ds + \sqrt{2}\bfB\right)\end{equation}
	where $\Gamma^N$ is the Skorokhod map associated with normal reflection, i.e. with reflection matrix $M= I$.
	 The `unconstrained version' of \eqref{eq:atwref}, namely the unique weak solution of the equation
	$$d\bfU(t) = (-\bfa + be(\bfU(t))) dt + \sqrt{2} d\bfB(t), \; t \ge 0,$$
	in the special case $a=0$ and $b=1$, is the well known standard Atlas model \cite{Atlas1,Atlas2}.
	
	Let $\tilde \bfY(\cdot)$ be the ordered version of $\bfY(\cdot)$ with ties broken in the lexicographic order. In particular,
	$\tilde \bfY = (\tilde Y_1, \ldots \tilde Y_K)$ where $\tilde Y_1(t) \le \tilde Y_2(t) \le \cdots \le \tilde Y_K(t)$ a.s. for all $t\ge 0$. Let $\tilde Z_i \doteq \tilde Y_i - \tilde Y_{i-1}$, $1\le i \le K$, where $\tilde Y_0(t)\doteq 0$ for $t\ge 0$.
	Then, using Tanaka's formula, it is easy to verify that $\tilde \bfZ(\cdot)$ has the same distribution as the process $\bfZ(\cdot)$ given in Theorem \ref{convtodiff}.  A similar statement holds for the  system in \eqref{z0eq} with  general values of $a_i$;  in this case the equation \eqref{eq:atwref} is changed to one where the $i$-th smallest coordinate of $\bfY(t)$ at time instant $t$ gets the drift $-a_i$.

\end{remark}

It is easily checked that $\Zbf(\cdot)$ defines a strong Markov process. The next result gives conditions on the parameters $a_1,...,a_K$ such that a stationary distribution for $\Zbf(\cdot)$ exists and is unique, and further gives an explicit expression for the stationary distribution under these conditions as a product of Exponential distributions.

\begin{theorem}\label{statdistthm} The Markov process $\mathbf{Z}(\cdot)$ is positive recurrent if and only if for all $i \in [K]$, the $a_1, ..., a_K$ satisfy the condition
\begin{equation}\label{stabcon}
    \sum_{j=i}^Ka_j > 0, \,\,\,\,\,\,\,\forall i \in [K]. 
\end{equation}
Moreover, if the above condition is satisfied, the stationary distribution $\hat \pi$ takes the form $\hat \pi(dx) = \pi(x) dx$, where the density $\pi(x)$ of the stationary distribution has a product form and is given as: 
\begin{equation}\label{expstatdist}
    \pi(x) \doteq c_{\pi}e^{\eta'x}, \; x \in \RR_+^K,
\end{equation}
where $c_\pi \in (0,\infty)$ is the normalizing constant, and $\eta \in \mathbb{R}^K$ is defined by
\begin{equation}\label{etai}
    \eta_i = -\sum_{j=i}^Ka_j, \, i \in [K].
\end{equation}


\end{theorem}
As a corollary of the above result we have the following important special case of the above theorem which, in particular, gives the stability region for the limiting gaps in the MJSQ scheme discussed in the Introduction and identifies its stationary distribution.
\begin{corollary}\label{cor:spcas}
	Suppose that for some $a\in (0, \infty)$ and $b \in \RR$, $a_1 = a-b$ and $a_i = a$ for $i = 2, 3, \ldots K$.
	Then the Markov process $\mathbf{Z}(\cdot)$ is positive recurrent and has a product form stationary distribution if and only if $b \in (-\infty , aK)$. In this case the stationary distribution $\hat \pi^*$ takes the form $\hat \pi^*(dx) = \pi^*(x) dx$, where the density $\pi^*(x)$ of the stationary distribution is given as: 
	\begin{align}
    \begin{split}\label{expstatdistSp}
	    \pi^*(x) &\doteq c_{\pi^*} e^{-(aK-b)x_1} \prod_{i=2}^Ke^{-a(K+1-i)x_i} , \; x \in \RR_+^K \\
      & {\cb = c_{\pi^*}e^{bx_1} \prod_{i=1}^Ke^{-a(K+1-i)x_i} },
    \end{split}
    \end{align}
	where $c_{\pi^*} \in (0,\infty)$ is the normalizing constant.
\end{corollary}

The limit theorem  in Theorem \ref{convtodiff} gives convergence in distribution of the scaled queue-length gap processes over any compact time interval. {\cb The second form of equation \eqref{expstatdist} emphasizes the difference in steady state behavior under MJSQ versus uniform routing due to the choice of the MJSQ parameter $b$ (the case $b=0$ corresponds to random routing). In particular, we can see that, for $b > 0$, a larger magnitude of $b$ results in the length of the shortest queue being longer on average than under uniform routing. On the other hand, when $b < 0$,  we can interpret this as a model, perhaps of less practical significance, under which there is a marginal bias (the size of which is controlled by $b$) \textit{against} sending incoming jobs to the shortest queue. In this case, a greater magnitude of $b$ results in a shorter steady state queue-length for the shortest queue as compared to uniform routing.} 

The next result says that, under conditions, convergence also holds at $t=\infty$, namely the stationary distribution of $\hat \bfZ^n$ converges to that of $\bfZ$.


\begin{theorem}\label{statdistconv} Suppose that \eqref{stabcon} is satisfied. Then, there exists $n_0 \in \mathbb{N}$ such that for each $n \ge  n_0$, there exists a unique stationary distribution $\hat{\pi}^n$ for the Markov process $\hat{\Zbf}^n(\cdot)$, and moreover $\hat{\pi}^n \Rightarrow \hat\pi$ as $n \rightarrow \infty$, where $\hat \pi$ is the unique stationary distribution of $\bfZ(\cdot)$.
\end{theorem}

\subsection{Load balancing performance of MJSQ}\label{mjsqrem} 

In this section  we compare the performance of the MJSQ, PQS and JSQ policies.

\noindent
(i) \textbf{MJSQ v/s PQS: }
For $a, b >0$, $K \ge 2$ and $d \in [K]$, consider the model \eqref{ordqs} with choice of parameters  $a_i = a - \frac{b}{d}$ for $i = 1, ..., d$ and $a_i = a$ for $i = d+1,...,K$. On comparing routing probabilities, it follows that this system corresponds to a load sharing scheme in which, of the total arrival rate of $nK - \sqrt{n}Ka + \sqrt{n}b$ to the system, 
$$ \frac{K - Ka/\sqrt{n}}{K - Ka/\sqrt{n} + b/\sqrt{n}}, \; \mbox{ and (resp.)} \; \frac{b/\sqrt{n}}{K - Ka/\sqrt{n} + b/\sqrt{n}}$$
 fraction of jobs are routed at random, and (resp.) routed uniformly at random to the shortest $d$ queues in the system.
The case $d=1$ is the MJSQ policy while the case $d=K$ is the PQS policy discussed in the Introduction.
With the above choice of the parameters $a_i$, the drift vector $\boldsymbol{\rho}$ of the limiting diffusion $\Zbf(\cdot)$ takes the form
\begin{align*}
    \rho_1 = -a + \frac{b}{d},\;\; \rho_{d+1} = -\frac{b}{d}, \;  
    \rho_i = 0, \; i \in [K] \setminus \{1,d+1\}.
\end{align*}
It can be verified that the stability condition \eqref{stabcon} is satisfied if and only if $b \in (0,aK)$. 

Note that with the above choice of $\{a_i\}$, $\upp = aK-b$.
From Theorem \ref{statdistthm} the stationary density  of the reflected diffusion $\Zbf(\cdot)$ takes the form in \eqref{expstatdist} where the vector $\eta$ in the exponent
is given as
\begin{equation}\label{etad}
   \eta_i = -(K-i+1)a + \frac{(d-i+1)(aK-\upp)}{d}\mathbf{1}_{\{i \le d\}} ,\,\,\,\,\,\,\,\,\, i \in [K].
\end{equation}
Denote the probability distribution with the  density given by \eqref{expstatdist} with the above expression of $\eta$ as $\hat \pi^d$.
Consider the total (limiting rescaled) queue-length process $W^d(\cdot) \doteq \sum_{i=1}^K X_i(\cdot)$, where $X_i(\cdot) \doteq X_{i-1}(\cdot) + Z_i(\cdot)$, for $i \in [K]$, with $X_0(\cdot) \doteq 0$ by convention. Note that this process gives the total {\cb queue-length } in the system over time. Also define the {\em system imbalance index}  $D^d(\cdot) \doteq X_K(\cdot) -X_1(\cdot) = \sum_{i=2}^{K}Z^d_i(\cdot)$ which gives the difference between the largest and the smallest queue lengths over time. It then follows  that
\begin{align*}
    \mathbb{E}_{\hat\pi^d}W^d(0) &= \frac{K-d}{a} + \sum_{j=1}^d \frac{(K-j+1)d}{(K-j+1)ad - (d-j+1)(aK-\upp)}, \\
    \mathbb{E}_{\hat\pi^d}D^d(0) &= \sum_{j=2}^d \frac{d}{(K-j+1)ad - (d-j+1)(aK-\upp)} + \sum_{j=d+1}^K \frac{1}{(K-j+1)a},
\end{align*}
where $\mathbb{E}_{\hat\pi^d}$ is the expectation under the stationary measure, namely under the probability measure $\mathbb{P}_{\hat\pi^d}$ on $(\Omega, \mathscr{F})$ such that $\mathbb{P}_{\hat\pi^d}\{\mathbf{Z}(0) \in A\} = \hat\pi^d(A)$ for $A \in \mathcal{B}(\mathbb{R}_+^K)$, and the appropriate sums above are taken to be zero if $d=1$ or $d=K$. Now define the ratios 
$$R_D \doteq \mathbb{E}_{\hat\pi^K}D^K(0)/\mathbb{E}_{\hat\pi^1} D^1(0), \; R_W \doteq \mathbb{E}_{\hat\pi^K}W^K(0)/\mathbb{E}_{\hat\pi^1} W^1(0).$$ 
The quantity $R_D$ quantifies the load imbalance in PQS in comparison to the MJSQ system while the quantity
$R_W$ measures the ratio between the limiting steady-state total {\cb queue-lengths} in PQS and MJSQ systems.
It can be checked that
\begin{align}
	\mathbb{E}_{\hat\pi^K}W^K(0) &= \frac{K^2}{\upp}, \ \ \mathbb{E}_{\hat\pi^1} W^1(0) = \frac{K-1}{a} + \frac{K}{\upp} ,\\
	\mathbb{E}_{\hat\pi^K}D^K(0) &= \frac{1}{\upp}\sum_{j=2}^K\frac{K}{K-j+1}, \ \ \mathbb{E}_{\hat\pi^1} D^1(0) = \frac{1}{a}\sum_{j=2}^K\frac{1}{K-j+1}.\label{eq:409n}
\end{align}
Thus,
\begin{align*}
    R_W & = \frac{K^2a}{(K-1)\upp + Ka} \approx \frac{Ka}{\upp + a},\\
    R_D & = \frac{1}{\upp}\sum_{j=2}^K\frac{K}{K-j+1}/\frac{1}{a}\sum_{j=2}^K\frac{1}{K-j+1} = \frac{Ka}{\upp}.
\end{align*}
These results say in particular that when $K$ is large, then in comparison with the standard PQS, the MJSQ system achieves (in the heavy traffic limit) a reduction in steady state mean total system  {\cb total queue-length} by approximately a factor of $Ka/(\upp + a)$. Moreover, in the steady state heavy traffic limit, the average system imbalance index in the MJSQ model is once more reduced by a factor of $Ka/\upp$ in comparison to the standard PQS. These results give quantitative measures of load balancing improvements achieved by the MJSQ policy over the PQS.
\\

\noindent (ii) \textbf{MJSQ v/s JSQ: }
Comparing the MJSQ policy with the JSQ policy, it is intuitively clear that the performance of the latter should be better. Moreover, by \cite[Theorem 5(b)]{chen2012asymptotic}, the JSQ policy is \emph{asymptotically optimal} in the sense that the average queue-length process (scaled by $\sqrt{n}$) is stochastically minimized in a certain sense as $n \rightarrow \infty$ amongst all feasible routing policies with $K$ servers. However, surprisingly, from \cite[Table 1]{chen2012asymptotic}, if we consider a JSQ policy with $K$ servers and the same net arrival rate $nK - \upp\sqrt{n}$ and service rate $n$ for each server as for the MJSQ policy discussed above, then the expected steady state limiting average queue-length (equivalently, average workload) for JSQ (as $n \rightarrow \infty$) is $1/\upp$ compared to $(K-1)/(Ka) + 1/\upp$ for the MJSQ policy obtained above. Thus, for large $n$, \emph{the average JSQ queue-length is smaller only by a constant factor $\approx (1+ \upp/a)$ than that of MJSQ, compared to a factor $K$ improvement of JSQ over PQS}. Moreover, for fixed $\upp$, as $a$ increases, $(1+ \upp/a)$ approaches one.
Also, from the state space collapse results of \cite{chen2012asymptotic} it follows that for the JSQ policy the expected system imbalance index (i.e. the quantity analogous to \eqref{eq:409n}) is zero. Thus from \eqref{eq:409n} we see that the price, in terms of the system imbalance, in going from JSQ to MJSQ is approximately $\log K/a$, which approaches $0$ as $a \to \infty$.
 \emph{ Thus the MJSQ policy becomes asymptotically optimal as $a \rightarrow \infty$}, in the sense that its performance, as measured by the steady state average queue-length and system imbalance index, approaches that of the JSQ policy.
   Thus, for large $K$, by implementing MJSQ, we achieve a \emph{significant performance improvement over PQS (comparable to that of JSQ) at approximately $1/\sqrt{n}$  of the communication cost incurred in implementing JSQ}. 
\\

\subsection{Overloaded System: Long Time Behavior}
\label{unstabsec}

Finally, we return to the subset of  parameter space considered in Corollary \ref{cor:spcas}: namely, where $a_1 = a-b$ and $a_i=a$ for $i = 2, \ldots, K$.
From Remark \ref{rem:atl} we see that in this case the unconstrained version $\bfU$ of $\bfY$ corresponds to the Atlas model where the lowest particle gets the drift $-a+b$ whereas the remaining particles get the drift $-a$.
We will consider the case where $a \ge 0$ and $b>aK$.
Note that in this case the necessary and sufficient condition for stability given in Corollary \ref{cor:spcas} fails to hold and therefore $ \bfZ$ is not positive recurrent. Similarly it can be checked that, in this case, $\bfZ^n$, and therefore the queue-length process $\bfQ^n$, is not stable for  large $n$. Nevertheless, it is of interest to investigate if the MJSQ policy has any quantifiable load balancing properties. 
  %
 %
 {\cb Writing $\hat\bfZ^n_c = (\hat Z^n_2, ..., \hat Z^n_K)'$},
 our final result says that in this regime (namely when $a\ge 0$ and $b>aK$)  the gap process $\hat \bfZ^n_c$ is stable for $n$ sufficently large. Specifically, although the full system $\hat \Zbf^n$ is unstable in this case and therefore does not have a stationary distribution, the result given below shows that $\hat \bfZ^n_c(t)$ does converge in distribution as $t \to \infty$
 and the limit, as $n\to \infty$, of these limiting distributions, takes an explicit product form. In other words, despite the system instablity, the discrepancy between the queue-length processes remains tight over time, demonstrating a novel form of `load balancing'.
 %
%
\begin{theorem}\label{unstabstatdist}
Fix $K \ge 2$. Suppose that for some $a \ge 0$ and $b>aK$,
$a_1 = a-b$ and $a_i=a$ for $i = 2, \ldots, K$. Then the following conclusions hold,

\begin{enumerate}
    \item[i.] There exists a $n_0 \in \NN$ and, for all $n \geq n_0$, $\RR^{K-1}_+$ valued random variables $\hat\Zbf^n_c(\infty)$ such that for each {\cb $m \ge n_0$, $\hat\Zbf^m_c(t) \Rightarrow \hat\Zbf^m_c(\infty)$ as $t\to \infty$.}
    \item[ii.] As $n \rightarrow \infty$, $\hat \Zbf_c^n(\infty) \Rightarrow \Zbf_c(\infty) = (Z_2(\infty),\dots, Z_K(\infty))'$, where 
    \begin{equation}\label{zcstat}
    \Zbf_c(\infty) \sim \bigotimes_{i=1}^{K-1}\textnormal{Exp}\left((\frac{K-i}{K})b\right).
    \end{equation}
\end{enumerate}

\end{theorem}
Note that the right side of \eqref{zcstat} is exactly the stationary distribution of the gap processes for a $K$ particle standard Atlas model where the lowest particle gets the drift $b/2$ (or $(-a+b)/2$) and remaining particles get the drift $0$ (resp. $-a/2$). This is natural to expect since in the regime of Theorem \ref{unstabstatdist} the lowest particle eventually escapes the origin and so the constrained diffusion behaves exactly as the standard Atlas model after a (random) finite period of time. 
We remark that in the case $a \in (0,\infty)$ and $b< aK$, the first part of Theorem \ref{unstabstatdist} can be shown to hold in exactly the same manner. However, in this case the  distributional limit of $\Zbf^n_c(\infty)$ is given by the (appropriate marginal distribution of the) probability measure $\hat \pi^*$ in Corollary \ref{cor:spcas} which is different from the right side of \eqref{zcstat}. In particular in the case $0<a <b <aK$, although the drift given to the lowest particle, namely $-a+b$, is strictly positive, it is not enough to overcome the influence of the higher particles to escape the origin eventually and as a consequence the boundary plays an important role in determining the time asymptotic distribution of $\Zbf_c^n$ and of its large $n$ process limit $\Zbf_c$. 

\begin{remark}\label{unload}
Observe that the expected long-time difference between the  maximum and minimum processes $\sum_{i=2}^KZ_i(\infty)$ in the unstable regime of Theorem \ref{unstabstatdist} is
\begin{equation}
    D_{unst} \doteq \mathbb{E}\left(\sum_{i=2}^KZ_i(\infty)\right) = \sum_{i=1}^{K-1}\frac{K}{b(K-i)}. \label{eq:424n}
\end{equation}
In comparison, from Corollary \ref{cor:spcas}, the same quantity in the stable regime when $a > 0$ and $b \in (-\infty, aK)$ takes the form
\begin{equation*}
    D_{stab} \doteq \frac{1}{a}\sum_{i=1}^{K-1}\frac{1}{K-i}.
\end{equation*}
From this, we can see that, for fixed $a,K$, as $b$ grows, the steady state system imbalance remains constant ($= D_{stab}$) when $b$ is less than $aK$. As $b$ grows above $aK$ (that is, we switch from the stable to the unstable regime), the system imbalance decays like $1/b$ and hence, \emph{although the whole system is unstable, load is more equally balanced.}

Also note that for the PQS policy defined as in Section \ref{mjsqrem} (for $d=K$), the $K$ queue-length processes behave as independent $M/M/1$ queues with arrival rate $n + (bK^{-1} - a)\sqrt{n}$ and departure rate $n$. Thus, when $b > aK$, each queue-length process is transient and it follows from their mutual independence that the difference between their maximum and minimum (system imbalance index) is not tight in time and the quantity analogous to that in  \eqref{eq:424n} is $\infty$ in this case.
In this sense, although the whole system is not stable in both the PQS and MJSQ policy for the range of parameters in Theorem \ref{unstabstatdist}, \emph{the latter still `balances load' by keeping the discrepancies between queue-lengths stable in time}.
\end{remark}

{\cb We now summarize the quantitative comparisons of long time behavior (asymptotically as $n \rightarrow \infty$) given above for MJSQ, JSQ and PQS in table form.

\begin{center}
\begin{tabular}
{ |c| c | c | c | c | } 

  \hline
  \, & Avg. queue-length &  Syst. imb. (stable) & Syst. imb. (unstable) & Approx. commun. cost\\ 
  \hline
  \textbf{MJSQ}& $\frac{K-1}{Ka}+\frac{1}{v}$ & $\frac{1}{a}\sum_{i=1}^{K-1}\frac{1}{K-i}$ & $\sum_{i=1}^{K-1}\frac{K}{b(K-i)}$ & $b\sqrt{n}K$\\ 
  \hline
  \textbf{PQS} & $\frac{K}{v}$ & $\frac{K}{v}\sum_{i=1}^{K-1}\frac{1}{K-i}$ & $\infty$ & $\approx 0$ \\ 
  \hline
  \textbf{JSQ} & $\frac{1}{v}$ & 0 & 0 & $nK^2$ \\ 
  \hline
\end{tabular}
\end{center}

} 

Lastly, we note that although the case $b=aK$ is not investigated here, we believe that the average long time behavior of the corresponding MJSQ system should resemble the unstable case (Theorem \ref{unstabstatdist}) where $b>aK$. This is because, although the lowest queue will become empty infinitely often, the time spent at zero will be small and thus the average boundary effect of the origin will diminish over time.

\subsection{Organization}\label{organization}

The rest of this article is organized as follows. In Section \ref{secheavtraf}, we prove Theorem \ref{convtodiff}. In Section \ref{seczstat}, we establish Theorem \ref{statdistthm}. In Section \ref{secconvstat}, we give the proof of Theorem \ref{statdistconv}. In Section \ref{secunstab}, we prove Theorem \ref{unstabstatdist}.

\section{Convergence in distribution of diffusion-scaled gap processes}\label{secheavtraf}

In this section we prove Theorem \ref{convtodiff} which says that
 the diffusion-scaled process $\hat{\Zbf}^n(\cdot)$ converges in distribution to a reflecting diffusion. To establish this result, we rewrite \eqref{ordqs} in terms of a Skorokhod  map applied to a martingale input with drift (and some additional error terms). We use a well-known functional central limit theorem for Poisson processes and some tightness results for the local times of various processes related to $\hat{\Zbf}^n(\cdot)$ to show that this input converges in distribution to a Brownian motion with drift. Finally, we exploit the continuity properties of the Skorokhod map to show that $\hat\Zbf^n(\cdot)$ converges in distribution to a reflecting Brownian motion with drift as given in the statement of Theorem \ref{convtodiff}.

 In the process, we will also establish an estimate for the second moment of the local time of $\hat{\mathbf{Z}}^n(\cdot)$ (Corollary \ref{pairwisecorr}) which will be used in Section \ref{secconvstat}.

We begin by establishing an important   tightness property.
\begin{lemma}\label{qtightprop}
	Let $Q^n_i(t)$ denote the queue-length of the $i$-th labeled queue at time $t$ when $Q^n_i(0) = X^n_i(0)$, $i \in [K]$.
	Suppose that the collection $\{\hat \bfZ^n(0), n \in \NN\}$ is tight.
Define $U^n_i(t) \doteq \sqrt{n}\int_0^{t} \mathbf{1}_{\{Q^n_i(s) = 0\}} ds$, $t\ge 0$.
Then  the collection $\{U^n_i, i \in [K], n \in \NN\}$ is tight in 
 $\mathcal{C}([0,\infty):\mathbb{R})$.
 Furthermore, there is a $D_0 \in (0, \infty)$ such that for all $T \in (0, \infty)$ and $n\in \NN$,
 $$ \max_{i \in [K]} \EE|U^n_i(T)|^2 \le D_0(1 + \EE(|\hat \bfZ^n(0)|^2) + T^2).$$
\end{lemma}
\begin{proof}
Let $\{A_i, D_i, i \in [K]\}$ be a collection of $2K$ independent rate $1$ Poisson process, independent of $\{Q^n_i(0), n \in \NN, i \in [K]\}$.  Recall the ranking map $r_i$ for $i \in [K]$ defined in Section \ref{sec:notat}. Then the process $\{Q^n_i(\cdot), i \in [K]\}$ has the following distributionally equivalent representation: 
\begin{align}
\begin{split}\label{unordqs}
    Q^n_i(t) = Q^n_i(0) + A_i\left(\int_0^t \kappa^n_i(s)ds\right) - D_i\left(n\int_0^t \mathbf{1}_{\{Q^n_i(s) > 0\}}ds\right),\\
\end{split}    
\end{align}
where for $s \geq 0$, 
\begin{equation}\label{kappadrift}
    \kappa^n_i(s) \doteq \sum_{j=1}^K (n - a_j\sqrt{n}) \mathbf{1}_{\{r_j(\Qbf^n(s)) =i\}},
\end{equation}
with $\Qbf^n(s) =(Q^n_1(s), \ldots , Q^n_K(s))'$.
For $t \geq 0$, writing $\hat{\Qbf}^n(t) \doteq \frac{1}{\sqrt{n}}\Qbf^n(t)$, the above can be rewritten as
\begin{align*}
    \hat{Q}^n_i(t)  = \hat{Q}^n_i(0) + M_i^n(t) + \frac{1}{\sqrt{n}} \int_0^t \kappa^n_i(s) ds - \sqrt{n} \int_0^t \mathbf{1}_{\{Q^n_i(s) > 0\}}ds
\end{align*}
where
\begin{align*}
	M^n_i(t) &\doteq  M^n_{A,i}(t) - M^n_{D,i}(t), \\
    M^n_{A,i}(t) & \doteq \frac{A_i\left(\int_0^t \kappa^n_i(s)ds\right) - \int_0^t \kappa^n_i(s)ds}{\sqrt{n}}, \\
    M^n_{D,i}(t) & \doteq \frac{D_i\left(n\int_0^t \mathbf{1}_{\{Q^n_i(s) > 0\}}ds\right) - n\int_0^t \mathbf{1}_{\{Q_i(s) > 0\}}ds}{\sqrt{n}}.
\end{align*}
Note that
$$\frac{1}{\sqrt{n}}\kappa^n_i(s) = \sqrt{n}\sum_{j=1}^K \mathbf{1}_{\{r_j(\Qbf^n(s)) =i\}} - \sum_{j=1}^Ka_j\mathbf{1}_{\{r_j(\Qbf^n(s)) =i\}} = \sqrt{n} - \sum_{j=1}^Ka_j\mathbf{1}_{\{r_j(\Qbf^n(s)) =i\}}$$
and $\sqrt{n}\mathbf{1}_{\{Q^n_i(s) > 0\}} = \sqrt{n} - \sqrt{n}\mathbf{1}_{\{Q^n_i(s) = 0\}}$.
Thus
\begin{align}
\begin{split}\label{skoroq}
    \hat{Q}^n_i(t) & = \hat{Q}^n_i(0) + M_i^n(t) - \sum_{j = 1}^Ka_j\int_0^t \mathbf{1}_{\{r_j(\Qbf^n(s)) =i\}}ds + \sqrt{n}\int_0^t\mathbf{1}_{\{\hat{Q}^n_i(s) = 0\}}ds \\
    & = \Gamma^N_{2,i}(\mathscr{Q}^n(\cdot))(t),
\end{split}
\end{align}
where $\mathscr{Q}^n(t) = (\mathscr{Q}^n_1(t), \ldots \mathscr{Q}^n_K(t))'$,
$$\mathscr{Q}^n_i(t)  = \hat{Q}^n_i(0) + M_i^n(t) - \sum_{j = 1}^Ka_j\int_0^{t} \mathbf{1}_{\{r_j(\Qbf^n(s)) =i\}}ds, \, i \in [K],$$
$\Gamma^N = (\Gamma^N_1, \Gamma^N_2)$ is the Skorokhod map with normal reflection, i.e. with the reflection matrix $M=I$
and for $\phi\in \mathcal{D}_0([0,\infty):\mathbb{R}^K)$, $\Gamma^N_{2,i}$ is the $i$-th coordinate of $\Gamma^N_2(\phi)$ (see Proposition \ref{prop:skor}). Furthermore, note that
\begin{equation}\label{eq:equnit}
	U^n_i(t) = \Gamma^N_{1,i}(\mathscr{Q}^n(\cdot))(t), \; t \ge 0,\end{equation}
where $\Gamma^N_{1,i}$ is the $i$-th coordinate of $\Gamma^N_1$. Thus, in view of the Lipschitz property of $\Gamma^N$, to prove the stated tightness in the lemma, it suffices to show that
$\{\mathscr{Q}^n, n \in \NN\}$ is tight in $\mathcal{D}([0,\infty):\mathbb{R}^K)$.
Note that by assumption, $\{\hat \Qbf^n(0), n \in \NN\}$ is a tight collection of $\RR^K_+$ valued random variables.
Also, the collection
$$\{\sum_{j = 1}^Ka_j\int_0^{\cdot} \mathbf{1}_{\{r_j(\Qbf^n(s)) =i\}}ds, \; i \in [K], n \in \NN\}$$
is clearly tight in $\mathcal{C}([0,\infty):\mathbb{R})$. Thus to prove the lemma, it suffices to show that
\begin{equation}\label{eq:marts}
	\{M^n_{A,i}, M^n_{D,i}, i \in [K], n \in \NN\}\end{equation}
is a tight collection of $\mathcal{D}([0,\infty):\mathbb{R})$ valued random variables. 
Note that for $i \in [K]$, the predictable quadratic variation $\langle M^n_{A,i}\rangle$ of the martingale $M^n_{A,i}$ satisfies
\begin{equation}\label{eq:qva}
	\langle M^n_{A,i}\rangle_t = \frac{1}{n} \int_0^t \kappa^n_i(s) \le (1+ |a|_*)t, \mbox{ for all } t \ge 0,\end{equation}
where $|a|_*\doteq \max_{l \in [K]} |a_l|$.
Similarly, for $i \in [K]$, the predictable quadratic variation $\langle M^n_{D,i}\rangle$ of the martingale $M^n_{D,i}$
satisfies
\begin{equation}\label{eq:qvd}\langle M^n_{D,i}\rangle_t = \int_0^t \mathbf{1}_{\{\hat{Q}^n_i(s) > 0\}}ds \le t, \mbox{ for all } t \ge 0.\end{equation}
The tightness of the collection in \eqref{eq:marts} is now immediate from the optional sampling theorem for martingales and the classical Aldous-Kurtz tightness criteria \cite{Aldous78}. 

We now prove the second statement in the lemma. From \eqref{eq:equnit} and the Lipschitz property of the Skorokhod map we have that, for $i \in [K]$ and $T \in (0, \infty)$,
\begin{align*}
	\EE(U^n_i(T))^2 &\le \EE\left(\sup_{0\le t \le T} |\hat{Q}^n_i(0) + M_i^n(t) - \sum_{j = 1}^Ka_j\int_0^{t} \mathbf{1}_{\{r_j(\Qbf^n(s)) =i\}}ds|^2\right)\\
	&\le 3\EE|\hat{Q}^n_i(0)|^2 + 3 \EE\left(\sup_{0\le t \le T}|M_i^n(t)|^2 + (K|a|_*T)^2\right),
\end{align*}
where we have used the fact that since $\Gamma^N$ corresponds to the Skorokhod map for normal reflection, it reduces to a collection of $K$ one-dimensional Skorokhod maps.  {\cb We have also used here (and at several later instances) the inequality
 $(\sum_{i=1}^k x_i)^2 \le k \sum_{i=1}^kx_i^2$, for any $k \in \NN$, $x_i \in \mathbb{R}$, $1\le i \le k$}.
Also, from Doob's inequality and the quadratic variation estimates in \eqref{eq:qva} and \eqref{eq:qvd}
$$\EE\sup_{0\le t \le T}|M_i^n(t)|^2 \le 2\EE\sup_{0\le t \le T}|M_{A,i}^n(t)|^2 + 2\EE\sup_{0\le t \le T}|M_{D,i}^n(t)|^2 \le 8(2+ |a|_*)T.$$
The result follows on combining the above estimates.
\end{proof}

The following lemma, which will be crucial to the proofs of Theorem \ref{convtodiff} and Theorem \ref{statdistconv}, says that the Lebesgue measure of the time the  system of queues spends, before any fixed time $T > 0$, with any two queues tied in length, converges to zero in probability as the heavy traffic parameter $n \rightarrow \infty$.
\begin{lemma}\label{pairwiselemma}
	Let $\{Q^n_i(\cdot), i \in [K], n\in \NN\}$ be as in the statement of Lemma \ref{qtightprop}.
Then for $i,j \in [K]$, $i \neq j$, and $T \in (0,\infty)$, as $n \rightarrow \infty$,
\begin{equation*}
    \int_0^{T} \mathbf{1}_{\{Q^n_i(s) = Q^n_j(s)\}} ds  \rightarrow 0, \; \mbox{ in probability.}
\end{equation*}
Furthermore, there is a $D_1 \in (0, \infty)$, such that for all $n \in \NN$, $T \in (0, \infty)$
and $i,j \in [K]$, $i\neq j$,
$$\EE \left(\int_0^{T} \mathbf{1}_{\{Q^n_i(s) = Q^n_j(s)\}} ds\right)^2 \le \frac{D_1}{n}(1+\EE(|\hat \bfZ^n(0)|^2) + T^2).$$
\end{lemma}

\begin{proof}
Let $i, j \in [K]$, $i\neq j$. For  $t \geq 0$, define $\hat{Z}_{i,j}^n(t) \doteq |\hat{Q}^n_i(t) - \hat{Q}^n_j(t)|$. Recall the definition \eqref{kappadrift} of $\kappa^n_l(\cdot)$, $l\in[K]$. Let $(A,D)$ be independent rate $1$ Poisson processes, independent of $\{Q^n_i(0), n \in \NN, i \in [K]\}$. Then (suppressing $n$ from the notation) the evolution of $\hat{Z}_{i,j}^n(t)$ has the following distributionally equivalent representation.
\begin{align*}
    \hat{Z}_{i,j}^n(t) & = \hat{Z}_{i,j}^n(0)+ \frac{1}{\sqrt{n}}A\left(\int_0^t\beta^{i,j}(s)ds\right) - \frac{1}{\sqrt{n}}D\left(\int_0^t\lambda^{i,j}(s)ds\right),
\end{align*}

where

\begin{align}
    \beta^{i,j}(s) &\doteq \kappa_i(s)\mathbf{1}_{\{Q_i(s) > Q_j(s) = 0\}} + (\kappa_i(s) + n)\mathbf{1}_{\{Q_i(s) > Q_j(s) > 0\}}\\
    & + \kappa_j(s)\mathbf{1}_{\{Q_j(s) > Q_i(s) = 0\}} + (\kappa_j(s)+n)\mathbf{1}_{\{Q_j(s) > Q_i(s) > 0\}} \nonumber\\
    & + (\kappa_i(s)+\kappa_j(s))\mathbf{1}_{\{Q_i(s) = Q_j(s) = 0\}} + (\kappa_i(s)+\kappa_j(s)+2n)\mathbf{1}_{\{Q_i(s) = Q_j(s) > 0\}}, \label{rev-one}\\
    \lambda^{i,j}(s) &\doteq (\kappa_j(s)+n)\mathbf{1}_{\{Q_i(s) > Q_j(s) = 0\}} + (\kappa_j(s) + n)\mathbf{1}_{\{Q_i(s) > Q_j(s) > 0\}}\nonumber\\
    & + (\kappa_i(s)+n)\mathbf{1}_{\{Q_j(s) > Q_i(s) = 0\}} + (\kappa_i(s)+n)\mathbf{1}_{\{Q_j(s) > Q_i(s) > 0\}}.\label{rev-two}
\end{align}
The terms $\beta^{i,j}$  can be understood as follows. On the set $\{Q_i(s-) > Q_j(s-) =0\}$, we have
$\hat{Z}_{i,j}^n(s) - \hat{Z}_{i,j}^n(s-) >0$ only if an arrival occurs to the $i$-th labeled queue which occurs at rate $\kappa_i(s)$.
Also, on the set $\{Q_i(s-) > Q_j(s-) >0\}$, $\hat{Z}_{i,j}^n(s) - \hat{Z}_{i,j}^n(s-) >0$  if either an arrival occurs  to the $i$-th labeled queue (which is at rate $\kappa_i(s)$) or a departure occurs from the $j$-th queue (which occurs at rate $n$). This explains the first two terms in the definition of $\beta^{i,j}(s)$. The next two terms are understood in a similar fashion on reversing the roles of $i$ and $j$.
Next, on the set $\{Q_i(s-) = Q_j(s-) =0\}$, we have $\hat{Z}_{i,j}^n(s) - \hat{Z}_{i,j}^n(s-) >0$ if either an arrival occurs to the $i$-th queue (rate $\kappa_i(s)$) or to the $j$-th queue (rate $\kappa_j(s)$). Similarly, on the set $\{Q_i(s-) = Q_j(s-) >0\}$ , we have $\hat{Z}_{i,j}^n(s) - \hat{Z}_{i,j}^n(s-) >0$ if either an arrival occurs to the $i$-th queue (rate $\kappa_i(s)$) or to the $j$-th queue (rate $\kappa_j(s)$), or a departure occurs from the $i$-th queue (rate $n$) or from the $j$-th queue (rate $n$). This explains the last two terms in the expression for 
$\beta^{i,j}$ .  The various terms in the definition of $ \lambda^{i,j}(s)$ can be understood in a similar fashion.

Hence, we may write
\begin{align*}
    \hat{Z}_{i,j}^n(t) & = \hat{Z}_{i,j}^n(0)+ M^n_{i,j}(t) + \frac{1}{\sqrt{n}}\int_0^t (\beta^{i,j}(s) - \lambda^{i,j}(s)) ds,
	\end{align*}
	where
\begin{align*}
	M^n_{i,j}(t) & \doteq M^n_{A,i,j}(t) - M^n_{D,i,j}(t), \\	
    M^n_{A,i,j}(t) & \doteq \frac{A\left(\int_0^t\beta^{i,j}(s)ds\right) - \int_0^t\beta^{i,j}(s)ds}{\sqrt{n}}, \;
    M^n_{D,i,j}(t)  \doteq \frac{D\left(\int_0^t\lambda^{i,j}(s)ds\right) - \int_0^t\lambda^{i,j}(s)ds}{\sqrt{n}}.
 \end{align*}  
{\cb 
Write 
 \begin{align*}
	 \beta^{i,j}(s) - \lambda^{i,j}(s)
	 = \mu^{i,j}(s) + v^{i,j}(s)
\end{align*}
where
\begin{align*}
\mu^{i,j}(s) &= 	\Big(\beta^{i,j}(s) - \lambda^{i,j}(s)\Big) - \Big((\kappa_i(s)+\kappa_j(s))\mathbf{1}_{\{Q_i(s) = Q_j(s) = 0\}}+ (\kappa_i(s)+\kappa_j(s)+2n)\mathbf{1}_{\{Q_i(s) = Q_j(s) > 0\}}\Big),\\
v^{i,j}(s) &= \Big((\kappa_i(s)+\kappa_j(s))\mathbf{1}_{\{Q_i(s) = Q_j(s) = 0\}}+ (\kappa_i(s)+\kappa_j(s)+2n)\mathbf{1}_{\{Q_i(s) = Q_j(s) > 0\}}\Big).
\end{align*}}
Then we can write
\begin{align}\label{zhatmgle}
	\hat{Z}_{i,j}^n(t) & = \hat{Z}_{i,j}^n(0)+ M^n_{i,j}(t) + \frac{1}{\sqrt{n}}\int_0^t \mu^{i,j}(s) ds + \frac{1}{\sqrt{n}}\int_0^t v^{i,j}(s) \mathbf{1}_{\{\hat{Z}_{i,j}^n(s)=0\}} ds\nonumber\\
	&= \Gamma_2^1(\mathscr{Z}^n_{i,j})(t),
\end{align}
where $\Gamma^1$ is the one-dimensional Skorokhod map on $\RR_+$ and 
$\mathscr{Z}^n_{i,j}(t) \doteq \hat{Z}_{i,j}^n(0)+ M^n_{i,j}(t) + \frac{1}{\sqrt{n}}\int_0^t\mu^{i,j}(s)ds$, for $t\ge 0$.

We now argue that the collection $\{\mathscr{Z}^n_{i,j}, n \in \NN\}$ is tight in $\mathcal{D}([0,\infty):\mathbb{R})$. Note that the collection $\{\mathscr{Z}^n_{i,j}(0), n \in \NN\}$ is tight by assumption.
Consider now the term $\frac{1}{\sqrt{n}}\int_0^{\cdot}\mu^{i,j}(s)(s)ds$.

{\cb Note that, from the definition of $\beta^{i,j}$ in \eqref{rev-one},
\begin{align*}
	&\beta^{i,j}(s) - \Big((\kappa_i(s)+\kappa_j(s))\mathbf{1}_{\{Q_i(s) = Q_j(s) = 0\}}+ (\kappa_i(s)+\kappa_j(s)+2n)\mathbf{1}_{\{Q_i(s) = Q_j(s) > 0\}}\Big)\\
	&= \kappa_i(s)\mathbf{1}_{\{Q_i(s) > Q_j(s) = 0\}} + (\kappa_i(s) + n)\mathbf{1}_{\{Q_i(s) > Q_j(s) > 0\}}\\
    &\quad + \kappa_j(s)\mathbf{1}_{\{Q_j(s) > Q_i(s) = 0\}} + (\kappa_j(s)+n)\mathbf{1}_{\{Q_j(s) > Q_i(s) > 0\}}.
\end{align*}
Thus, using the definition of $\lambda^{i,j}$ from \eqref{rev-two},
\begin{align*}
\mu^{i,j}(s) &= 	\kappa_i(s)\mathbf{1}_{\{Q_i(s) > Q_j(s) = 0\}} + (\kappa_i(s) + n)\mathbf{1}_{\{Q_i(s) > Q_j(s) > 0\}}\\
    &\quad + \kappa_j(s)\mathbf{1}_{\{Q_j(s) > Q_i(s) = 0\}} + (\kappa_j(s)+n)\mathbf{1}_{\{Q_j(s) > Q_i(s) > 0\}} - \lambda^{i,j}(s)\\
	&= (\kappa_i(s)- \kappa_j(s)-n)\mathbf{1}_{\{Q_i(s) > Q_j(s) = 0\}} + (\kappa_i(s)- \kappa_j(s))\mathbf{1}_{\{Q_i(s) > Q_j(s) > 0\}} \\
	&\quad + (\kappa_j(s)- \kappa_i(s)-n)\mathbf{1}_{\{Q_j(s) > Q_i(s) = 0\}} + (\kappa_j(s)- \kappa_i(s))\mathbf{1}_{\{Q_j(s) > Q_i(s) > 0\}}.
\end{align*}
It now follows that
\begin{align*}
    \frac{1}{\sqrt{n}}|\mu^{i,j}(s)| & = \frac{1}{\sqrt{n}}\Big|(\kappa_i(s) - \kappa_j(s)-n)\mathbf{1}_{\{Q_i(s) > Q_j(s) = 0\}} + (\kappa_i(s) - \kappa_j(s))\mathbf{1}_{\{Q_i(s) > Q_j(s) > 0\}}\\
    &\,\,\,\,\,\,\,\,\,\,\,\,\,\,\,\,\,\,\, - (\kappa_i(s) - \kappa_j(s)+n)\mathbf{1}_{\{Q_j(s) > Q_i(s) = 0\}} - (\kappa_i(s) - \kappa_j(s))\mathbf{1}_{\{Q_j(s) > Q_i(s) > 0\}}\Big|. 
\end{align*}
}
%
%
%
Define 
$$\Delta_{i,j}(\cdot) \doteq  \frac{1}{\sqrt{n}}(\kappa_i(\cdot)- \kappa_j(\cdot)) = \sum_{l=1}^Ka_l(\mathbf{1}_{\{r_l(\mathbf{Q}^n(\cdot) =j\}}-\mathbf{1}_{\{r_l(\mathbf{Q}^n(\cdot) =i\}}).$$
Recalling that  $|a|_*\doteq \max_{l \in [K]} |a_l|$, we now have that, for $n > |a|^2_*$, for each $s \geq 0$,
\begin{align*}
    \frac{1}{\sqrt{n}}|\mu^{i,j}(s)| 
    & = |(\Delta_{i,j}(s)-\sqrt{n})\mathbf{1}_{\{Q_i(s) > Q_j(s) = 0\}} + \Delta_{i,j}(s)\mathbf{1}_{\{Q_i(s) > Q_j(s) > 0\}}\\
    &\,\,\,\,\,\,\, - (\Delta_{i,j}(s)+\sqrt{n})\mathbf{1}_{\{Q_j(s) > Q_i(s) = 0\}} - \Delta_{i,j}(s)\mathbf{1}_{\{Q_j(s) > Q_i(s) > 0\}}| \\
    & \leq \sqrt{n}\mathbf{1}_{\{Q_j(s) = 0\}} + \sqrt{n}\mathbf{1}_{\{Q_i(s) = 0\}} + 2|a|_*.
\end{align*}
Consequently, for each $0 \le t_1 \le t_2$,
\begin{align}\label{eq:muij}
    \frac{1}{\sqrt{n}}\int_{t_1}^{t_2} |\mu^{i,j}(s)|ds &\leq  \sqrt{n}\int_{t_1}^{t_2}\mathbf{1}_{\{Q_j(s) = 0\}}ds + \sqrt{n}\int_{t_1}^{t_2}\mathbf{1}_{\{Q_i(s) = 0\}}ds + 2|a|_*(t_2-t_1).
\end{align}
The tightness of the family $\{\frac{1}{\sqrt{n}}\int_0^{\cdot}\mu^{i,j}(s)(s)ds, n \in \NN\}$ is now an immediate consequence of 
Lemma \ref{qtightprop}.

We now argue the tightness of the collection $\{M^n_{A,i,j}, M^n_{D,i,j}, n \in \NN\}$.
Note that, for $n > |a|^2_*$, the predictable quadratic variation $\langle M^n_{A,i,j}\rangle$ of the martingale $M^n_{A,i,j}$ satisfies
\begin{equation}\label{eq:qvmnaij}
	\langle M^n_{A,i,j}\rangle_t = \frac{1}{n} \int_0^t \beta^{i,j}(s) ds \le \frac{1}{n} \int_0^t(\kappa_i(s)+ \kappa_j(s)+2n) ds
\le 6t,\end{equation}
and the predictable quadratic variation $\langle M^n_{D,i,j}\rangle$ of the martingale $M^n_{D,i,j}$ satisfies
\begin{equation}\label{eq:qvmndij}
	\langle M^n_{D,i,j}\rangle_t = \frac{1}{n} \int_0^t \lambda^{i,j}(s) ds \le \frac{1}{n} \int_0^t(\kappa_i(s)+ \kappa_j(s)+n) ds
\le 5t.\end{equation}
The tightness of $\{M^n_{A,i,j}, M^n_{D,i,j}, n \in \NN\}$ is now immediate from standard tightness criteria.

Combining this with the tightness of $\{\frac{1}{\sqrt{n}}\int_0^{\cdot}\mu^{i,j}(s)ds, n \in \NN\}$ shown previously we now have the tightness of  $\{\mathscr{Z}^n_{i,j}(\cdot), n \in \NN\}$.
From the Lipschitz continuity of the map $\Gamma^1_1$ we now immediately get the tightness of $\{\zeta^n_{i,j}(\cdot), n \in \NN\}$, where
$$\zeta^n_{i,j}(t) \doteq \Gamma^1_1(\mathscr{Z}^n_{i,j})(t) 
= \frac{1}{\sqrt{n}}\int_0^t v^{i,j}(s) \mathbf{1}_{\{\hat{Z}_{i,j}^n(s)=0\}}ds, \; t \ge 0.$$
Finally, noting that for $n > 4|a|^2_*$, for all $s\ge 0$, on the event $\hat{Z}_{i,j}^n(s)=0$,
$$v^{i,j}(s) \ge \kappa_i(s) + \kappa_j(s) \ge \frac{n}{2} + \frac{n}{2}= n,$$
we have, for $0 \le t_1 \le t_2$,
\begin{align}\label{qijltineq}
\sqrt{n}  \int_{t_1}^{t_2}\mathbf{1}_{\{Q^n_i(s) = Q^n_j(s)\}}ds 
\le \frac{1}{\sqrt{n}} \int_{t_1}^{t_2} v^{i,j}(s) \mathbf{1}_{\{\hat{Z}_{i,j}^n(s)=0\}}ds \le |\zeta^n_{i,j}(t_2)- \zeta^n_{i,j}(t_1)|.
\end{align}
This shows that $\{\sqrt{n}  \int_{0}^{\cdot}\mathbf{1}_{\{Q^n_i(s) = Q^n_j(s)\}}ds, n \in\NN\}$ is tight  in $\mathcal{D}([0,\infty):\mathbb{R})$ implying the  convergence stated in the lemma.

Now we consider the second statement in the lemma. 
{\cb Note that for the one dimensional Skorohod map we have the following Lipschitz property: For $f, g \in \cld([0, \infty): \RR)$
and $T>0$
\begin{multline*}
\sup_{0\le t \le T} |\Gamma^1(f)- \Gamma^1(g)| = \sup_{0\le t \le T} |(f(t) - \inf_{0\le s \le t} (f(s)\wedge 0)) -
(g(t) - \inf_{0\le s \le t} (g(s)\wedge 0))|\\ 
\le 2 \sup_{0\le t \le T} |f(t)-g(t)|.
\end{multline*}
From this and the inequality \eqref{qijltineq}}, for $i,j \in [K]$ and $T \in (0, \infty)$,
\begin{align*}
	n\EE\left(\int_0^T\mathbf{1}_{\{Q^n_i(s) = Q^n_j(s)\}}ds \right)^2 &\le  \EE|\zeta^n_{i,j}(T)|^2 \le  4\EE \sup_{0\le t \le T}|\mathscr{Z}^n_{i,j}(t)|^2\\
	&\le 12\EE(\hat{Z}_{i,j}^n(0))^2+ 12\EE(\sup_{0\le t \le T} |M^n_{i,j}(t)|)^2 + 12\EE\left(\frac{1}{\sqrt{n}}\int_0^T|\mu^{i,j}(s)|ds\right)^2\\
	&\le 12\EE(\hat{Z}_{i,j}^n(0))^2+ 12\EE(\sup_{0\le t \le T} |M^n_{i,j}(t)|)^2 \\
	&\quad+
	72\max_{l\in [K]}\EE\left(\sqrt{n} \int_0^T \mathbf{1}_{\{Q_l(s)=0\}}ds\right)^2 + 144 |a|^2_*T^2,
\end{align*}
where the last inequality follows from \eqref{eq:muij}.
Using Lemma \ref{qtightprop} we get that
$$\max_{l\in [K]}\EE\left(\sqrt{n} \int_0^T \mathbf{1}_{\{Q_l(s)=0\}}\right)^2 \le 
D_0(1+ \EE|\hat \bfZ^n(0)|^2 + T^2).$$
Also,
\begin{align*}
	\EE(\sup_{0\le t \le T} |M^n_{i,j}(t)|)^2 \le 2 \EE(\sup_{0\le t \le T} |M^n_{A,i,j}(t)|)^2 +
	2 \EE(\sup_{0\le t \le T} |M^n_{D,i,j}(t)|)^2 \le 88T,
\end{align*}
where the last inequality is from Doob's inequality and estimates in \eqref{eq:qvmnaij} and \eqref{eq:qvmndij}.
The result follows on combining the last three estimates.
\end{proof}

As an immediate consequence of the above lemma we have the following  corollary on ties between ordered queues.
\begin{corollary}\label{pairwisecorr}
    For $i,j\in [K]$, $i\neq j$, and $T \in (0, \infty)$, as $n \rightarrow \infty$,
\begin{equation*}
    \int_0^{T} \mathbf{1}_{\{X^n_i(s) = X^n_j(s)\}} ds \rightarrow 0, \mbox{ in probability}.
\end{equation*}
Furthermore, there is a $D_2 \in (0, \infty)$ such that for every $T\in (0, \infty)$, $n \in \NN$ and
$i,j\in [K]$, $i\neq j$,
$$\EE \left(\int_0^{T} \mathbf{1}_{\{X^n_i(s) = X^n_j(s)\}} ds\right)^2 \le \frac{D_2}{n}(1+\EE(|\hat \bfZ^n(0)|^2) + T^2).$$
\end{corollary}

\begin{proof}
The result is immediate from Lemma \ref{pairwiselemma}  and the inequality 
\begin{align*}
    \int_0^T \mathbf{1}_{\{X^n_i(s) = X^n_j(s)\}} ds \leq {\cb \sum_{k,l \in [K]: k < l}\int_0^t \mathbf{1}_{\{Q^n_k(s) = Q^n_l(s)\}} ds.}
\end{align*}
{\cb Indeed, there are ${K\choose 2}$ terms in the sum and the statement of Lemma \ref{pairwiselemma} applies to each term in the sum.}
\end{proof}

We now proceed to the proof of Theorem \ref{convtodiff}.\\

\noindent\textit{Proof of Theorem \ref{convtodiff}}.
From \eqref{ordqs} we can write, for $i \in [K]$ and $t\ge 0$,
\begin{align*}
\hat{X}^n_i(t) = \hat{X}^n_i(0)	+ \tilde M^n_i(t) + \sqrt{n} \int_0^t \alpha^n_i(s)\mathbf{1}_{\{X^n_i(s) < X^n_{i+1}(s)\}}ds
-\sqrt{n} \int_0^t \delta^n_i(s)\mathbf{1}_{\{X^n_{i}(s) > X^n_{i-1}(s)\}}ds,
\end{align*}
where
\begin{align*}
	\tilde M^n_i(t) & \doteq  \tilde M^n_{A,i}(t) - \tilde M^n_{D,i}(t),\\
    \tilde M^n_{A,i}(t) & \doteq \frac{A_i\left(n\int_0^t \alpha^n_i(s)\mathbf{1}_{\{X^n_i(s) < X^n_{i+1}(s)\}}ds\right)-n\int_0^t \alpha^n_i(s)\mathbf{1}_{\{X^n_i(s) < X^n_{i+1}(s)\}}ds}{\sqrt{n}},\\
    \tilde M^n_{D,i}(t) & \doteq \frac{D_i\left(n\int_0^t \delta^n_i(s)\mathbf{1}_{\{X^n_{i}(s) > X^n_{i-1}(s)\}}ds\right) - n\int_0^t \delta^n_i(s)\mathbf{1}_{\{X^n_{i}(s) > X^n_{i-1}(s)\}}ds}{\sqrt{n}}.
\end{align*}
Note that 
\begin{equation}\label{ad}
\alpha^n_i(t) = \delta^n_i(t) - \frac{1}{\sqrt{n}}\sum_{j=1}^K a_j \mathbf{1}_{\{X^n_j(t) = X^n_{i}(t)\}}.
\end{equation}
Thus, letting
$$\hat{L}_{i}^n(t)  = (2-\mathbf{1}_{\{1\}}(i))\sqrt{n}\int_0^t \delta_{i}(s)\mathbf{1}_{\{X^n_i(s) = X^n_{i-1}(s)\}},$$
where $1_{\{1\}}:\mathbb{N} \rightarrow \{0,1\}$ is the indicator of the singleton $\{1\}$, we have, for $i \in [K]$,
\begin{align*}
\sqrt{n} \int_0^t \delta^n_i(s)\mathbf{1}_{\{X^n_i(s) < X^n_{i+1}(s)\}}ds &= \sqrt{n} \int_0^t \delta^n_i(s)ds
	- \sqrt{n} \int_0^t \delta^n_i(s)\mathbf{1}_{\{X^n_i(s) = X^n_{i+1}(s)\}}ds\\
&= \sqrt{n} \int_0^t \delta^n_i(s)ds
	- \sqrt{n} \int_0^t \delta^n_{i+1}(s)\mathbf{1}_{\{X^n_i(s) = X^n_{i+1}(s)\}}ds\\
	&= \sqrt{n} \int_0^t \delta^n_i(s)ds - \frac{1}{2}\hat{L}_{i+1}^n(t),
\end{align*}
and
\begin{align*}
\sqrt{n} \int_0^t \delta^n_i(s)\mathbf{1}_{\{X^n_{i}(s) > X^n_{i-1}(s)\}}ds &= 	\sqrt{n} \int_0^t \delta^n_i(s)ds - \sqrt{n} \int_0^t \delta^n_i(s)\mathbf{1}_{\{X^n_{i}(s) = X^n_{i-1}(s)\}}ds\\
&= \sqrt{n} \int_0^t \delta^n_i(s)ds - \frac{1+\mathbf{1}_{\{1\}}(i)}{2}\hat{L}^n_i(t).
\end{align*}
Combining these we have
\begin{align*}
	&\sqrt{n} \int_0^t \delta^n_i(s)\mathbf{1}_{\{X^n_i(s) < X^n_{i+1}(s)\}}ds -
	\sqrt{n} \int_0^t \delta^n_i(s)\mathbf{1}_{\{X^n_{i}(s) > X^n_{i-1}(s)\}}ds\\
	&= \frac{1+\mathbf{1}_{\{1\}}(i)}{2}\hat{L}^n_i(t) - \frac{1}{2}\hat{L}_{i+1}^n(t).
\end{align*}
Next note that
\begin{align*}
 &\int_0^t (\sum_{j=1}^K a_j \mathbf{1}_{\{X^n_j(s) = X^n_{i}(s)\}}) \mathbf{1}_{\{X^n_i(s) < X^n_{i+1}(s)\}}ds\\
 &=\int_0^t \sum_{j=1}^{i-1} a_j \mathbf{1}_{\{X^n_j(s) = X^n_{i}(s) <X^n_{i+1}(s)\}} ds
  + \int_0^t a_i\mathbf{1}_{\{X^n_i(s) < X^n_{i+1}(s)\}}ds\\
 &= \int_0^t \sum_{j=1}^{i-1} a_j \mathbf{1}_{\{X^n_j(s) = X^n_{i}(s) <X^n_{i+1}(s)\}} ds + t a_i - a_i \int_0^t \mathbf{1}_{\{X^n_i(s) = X^n_{i+1}(s)\}}ds.
\end{align*}
Combining the last two displays and using \eqref{ad} we see that
\begin{align*}
	&\sqrt{n} \int_0^t \alpha^n_i(s)\mathbf{1}_{\{X^n_i(s) < X^n_{i+1}(s)\}}ds
	-\sqrt{n} \int_0^t \delta^n_i(s)\mathbf{1}_{\{X^n_{i}(s) > X^n_{i-1}(s)\}}ds\\
	&= -a_i t + \mathcal{E}_i^n(t) + \frac{1+\mathbf{1}_{\{1\}}(i)}{2}\hat{L}^n_i(t) - \frac{1}{2}\hat{L}_{i+1}^n(t)
\end{align*}
where
$$
\mathcal{E}_i^n(t)  = a_i\int_0^t  \mathbf{1}_{\{X_i(s) = X_{i+1}(s)\}}ds - \sum_{j=1}^{i-1}a_j\int_0^t \mathbf{1}_{\{X_j(s) = X_i(s) < X_{i+1}(s)\}}ds.$$
Thus, for $i \in [K]$ and $t\ge 0$,
$$\hat{X}^n_i(t) = \hat{X}^n_i(0)	+ \tilde M^n_i(t) -a_i t + \mathcal{E}_i^n(t) + \frac{1+\mathbf{1}_{\{1\}}(i)}{2}\hat{L}^n_i(t) - \frac{1}{2}\hat{L}_{i+1}^n(t).$$
Letting $\tilde{\mathcal{E}}_i^n(t) \doteq \mathcal{E}_i^n(t)- \mathcal{E}_{i-1}^n(t)$, where ${\mathcal{E}}_{0}^n(t)\doteq 0$, and
taking differences, we now see that $\hat \bfZ^n$ satisfies:
\begin{align}
\begin{split}\label{z0eqn}
    \hat Z^n_1(t) & = \hat Z^n_1(0) + \tilde M^n_1(t) - a_1t + \tilde{\mathcal{E}}_1^n(t) - \frac{1}{2}\hat L^n_2(t) + \hat L^n_1(t) \\
    \hat Z^n_2(t) & = \hat Z^n_2(0) + (\tilde M^n_2(t) - \tilde M^n_1(t)) - (a_2 - a_1)t +\tilde{\mathcal{E}}_2^n(t)-\frac{1}{2}\hat L^n_3(t) + \hat L^n_2(t) - \hat L^n_1(t) \\
    \hat Z^n_i(t) & = \hat Z^n_i(0) + (\tilde M^n_i(t) - \tilde M^n_{i-1}(t)) - (a_i - a_{i-1})t + \tilde{\mathcal{E}}_i^n(t)-\frac{1}{2}(\hat L^n_{i+1}(t) + \hat L^n_{i-1}(t)) + \hat L^n_{i}(t),
\end{split}
\end{align}
for $i \in \{3,\ldots,K\}$.

Now define the $K$-dimensional vectors
\begin{align*}
    {\cb \tilde{\mathbf{M}}^n(t) \doteq \frac{1}{\sqrt{2}}( \tilde{M}^n_1(t), \ldots,  \tilde{M}^n_K(t))'},\; 
    \hat{\mathbf{L}}^n(t) \doteq (\hat{L}^n_1(t), \ldots, \hat{L}_K^n(t))',\;
    \tilde{\mathcal{E}}^n(t) \doteq (\tilde{\mathcal{E}}_1^n(t),\ldots,\tilde{\mathcal{E}}_K^n(t))'.
\end{align*}
Then the evolution equation for $\hat \bfZ^n(\cdot)$ can be written in vector form as
\begin{align}
	\hat \bfZ^n(t) &= \hat \bfZ^n(0) + \boldsymbol{\rho} t + A\tilde{\mathbf{M}}^n(t) + \tilde{\mathcal{E}}^n(t) + \clr \hat{\mathbf{L}}^n(t)\nonumber\\
	&= \Gamma_2(\hat \bfZ^n(0) + \boldsymbol{\rho} \id + A\tilde{\mathbf{M}}^n+ \tilde{\mathcal{E}}^n)(t),\label{eq:838}
\end{align}
where $\Gamma = (\Gamma_1,\Gamma_2)$ is the Skorokhod map associated with the reflection matrix $\clr$ and the last line follows on noting that
$\int_0^t \hat Z^n_i(s) d\hat L^n_{i}(s) =0$ for all $i \in [K]$ and $t\ge 0$.

Next note that for $i \in [K]$
 and $t \geq 0$,
\begin{align}
\begin{split}\label{errorbound}
    |\mathcal{E}^n_i(t)| & = \left|a_i\int_0^t  \mathbf{1}_{\{X^n_i(s) = X^n_{i+1}(s)\}}ds - \sum_{j=1}^{i-1}a_j\int_0^t \mathbf{1}_{\{X^n_j(s) = X^n_i(s) < X^n_{i+1}(s)\}}ds\right| \\
    & \leq \sum_{j = 1}^{i-1}|a_j|\int_0^t \mathbf{1}_{\{X^n_j(s) = X^n_i(s)\}}ds + |a_i|\int_0^t  \mathbf{1}_{\{X^n_i(s) = X^n_{i+1}(s)\}}ds.
\end{split}
\end{align}
From Corollary \ref{pairwisecorr}, it now follows that, as $n\to \infty$, $\mathcal{E}^n \Rightarrow 0$ and consequently
\begin{equation} \label{eq:843}
	\tilde{\mathcal{E}}^n \Rightarrow 0 
	\end{equation} 
	in $\mathcal{D}([0,\infty):\mathbb{R}^K)$.

Also note that, by the functional central limit theorem for Poisson processes, with
\begin{align*}
    \mathcal{M}^n_{A,D}(t) \doteq (\frac{A_1(nt) - nt}{\sqrt{n}},..., \frac{A_K(nt) - nt}{\sqrt{n}}, \frac{D_1(nt) - nt}{\sqrt{n}}, ..., \frac{D_K(nt) - nt}{\sqrt{n}})',\; t \ge 0,
\end{align*}
 we have that
\begin{equation}\label{eq:fclt}
    \mathcal{M}^n_{A,D}(\cdot) \Rightarrow\,\, \tilde{\mathbf{B}}(\cdot),
\end{equation}
in $\mathcal{D}([0,\infty):\mathbb{R}^{2K})$ where $\tilde{\mathbf{B}}(\cdot) = (\tilde B_1(\cdot),\dots,\tilde B_{2K}(\cdot))'$ is a standard $2K$-dimensional Brownian motion. 

Next, define, for $i \in [K]$,
\begin{align*}
    \Phi_{A,i}^n(t) \doteq \int_0^t \alpha^n_i(s)\mathbf{1}_{\{X^n_i(s) < X^n_{i+1}(s)\}}ds, \;
    \Phi_{D,i}^n(t) \doteq \int_0^t \delta^n_i(s)\mathbf{1}_{\{X^n_{i}(s) > X^n_{i-1}(s)\}}ds,
\end{align*}
%
Observe that, 
\begin{align*}
    \Phi_{A,i}^n(t) & = \int_0^t \alpha^n_i(s)\mathbf{1}_{\{X^n_i(s) < X^n_{i+1}(s)\}}ds \\
    & = \sum_{j = 1}^K (1 - \frac{a_j}{\sqrt{n}})\int_0^t \mathbf{1}_{\{X^n_j(s) = X^n_i(s) < X^n_{i+1}(s)\}}ds \\
    & = \sum_{j \in [K]: j \neq i} (1 - \frac{a_j}{\sqrt{n}})\int_0^t \mathbf{1}_{\{X^n_j(s) = X^n_i(s) < X^n_{i+1}(s)\}}ds + (1 - \frac{a_i}{\sqrt{n}})\int_0^t \mathbf{1}_{\{X^n_i(s) < X^n_{i+1}(s)\}}ds \\
    & = \sum_{j \in [K]: j \neq i} (1 - \frac{a_j}{\sqrt{n}})\int_0^t \mathbf{1}_{\{X^n_j(s) = X^n_i(s) < X^n_{i+1}(s)\}}ds + (1 - \frac{a_i}{\sqrt{n}})t - (1 - \frac{a_i}{\sqrt{n}})\int_0^t \mathbf{1}_{\{X^n_i(s) = X^n_{i+1}(s)\}}ds. 
\end{align*}
Thus for $n \ge |a|^2_*$ and $T \in (0, \infty)$,	
\begin{align*}
	\sup_{0\le t \le T}|\Phi_{A,i}^n(t) - t| \le \frac{a_i}{\sqrt{n}}T + \sum_{j \in [K]: j \neq i} \int_0^T \mathbf{1}_{\{X_j(s) = X_i(s) < X_{i+1}(s)\}}ds + \int_0^T \mathbf{1}_{\{X_i(s) = X_{i+1}(s)\}}ds.
\end{align*}
From Corollary \ref{pairwisecorr}, the right side in the above expression converges to $0$ in probability which shows that, for every $T \in (0, \infty)$, as $n \to \infty$,
$$\sup_{0\le t \le T}|\Phi_{A,i}^n(t) - t| \to 0 \mbox{ in probability}.$$
A very similar calculation shows that for every $T \in (0, \infty)$, as $n \to \infty$,
$$\sup_{0\le t \le T}|\Phi_{D,i}^n(t) - t| \to 0 \mbox{ in probability}.$$
Combining the above two facts with \eqref{eq:fclt} we have by a standard result (see \cite[Sec. 14]{BilConv}), that with 
\begin{align*}
\tilde \bfM^n_{A,D}(t) &\doteq ( (\tilde M^n_{A, i})_{i=1}^K, (\tilde M^n_{D, i})_{i=1}^K)\\
& =
\left( (n^{-1/2}(A_i(n\Phi_{A,i}^n(t)) - n\Phi_{A,i}^n(t) ))_{i=1}^K,   (n^{-1/2}(D_i(n\Phi_{D,i}^n(t))- n\Phi_{D,i}^n(t)))_{i=1}^K\right), \; t \ge 0,
\end{align*}
we have, as $n\to \infty$,
\begin{equation}
	\tilde\bfM^n_{A,D} \Rightarrow  \tilde{\mathbf{B}}
\end{equation}	
in  $\mathcal{D}([0,\infty):\mathbb{R}^{2K})$.
Letting
$$B_i \doteq \frac{1}{\sqrt{2}}(\tilde B_i - \tilde B_{i+K}), \; i \in [K],$$
and 
$\bfB = (B_1, \ldots B_K)'$, we now have that, as $n\to \infty$,
$$\tilde{\mathbf{M}}^n \Rightarrow \bfB$$
in  $\mathcal{D}([0,\infty):\mathbb{R}^{K})$. Note that $\bfB$ is a standard $K$-dimensional Brownian motion.

Combining this with \eqref{eq:838}, \eqref{eq:843}, and recalling that $\hat \bfZ^n(0) \to \bfZ(0)$ and the independence of $\{X_i^n(0)\}_{i \in [K]}$ from
$\{A_i, D_i, 1 \le i \le K\}$, we now have that  $\mathbf{Z}(0)$ is independent of  $\mathbf{B}$ and
$\hat \bfZ^n \Rightarrow \bfZ$, where 
$$\mathbf{Z}= \Gamma_2(\mathbf{Z}(0) + \boldsymbol{\rho}\,\id+ A\mathbf{B}),$$
where $\Gamma = (\Gamma_1, \Gamma_2)$ is the Skorokhod map associated with the reflection matrix $\clr$.
The result follows.
\hfill \qedsymbol

\section{Stationary distributions for $\Zbf$}\label{seczstat}

In this section, we will prove Theorem \ref{statdistthm} which gives a necessary and sufficient condition for $\bfZ$ to be positive recurrent and furthermore provides an explicit product form for this stationary distribution as a product of Exponential distributions. The proofs of these results rely on  \cite{harwil, harwil2,srbmwilliams}.\\

\noindent \textbf{Proof of Theorem \ref{statdistthm}.}

By \cite[Theorems 3.3]{srbmwilliams} (see also \cite{harwil}), the process $\mathbf{Z}$ is positive recurrent if and only if $\mathscr{R}^{-1}\rho < 0$.
We will now verify that this condition holds if and only if the inequalities \eqref{stabcon} are satisfied. 

Write $a_0 \doteq 0$. Observe that, for $i \in [K]$,
\begin{align}
\begin{split}\label{rhoid}
    \sum_{j=1}^i(i-j+1)(-\rho_j) & =  \sum_{j=1}^i(i-j+1)(a_j-a_{j-1}) \\
    & = \sum_{j=1}^i(i-j+1)a_j-\sum_{j=0}^{i-1}(i-j)a_j
     = \sum_{j=1}^ia_j.
\end{split}
\end{align}
Also, for $i = 2, ..., K$, using \eqref{rhoid}, we have that
\begin{align*}
    \sum_{j=i}^Ka_j & = \sum_{j=1}^Ka_j - \sum_{j=1}^{i-1}a_j\\
    & = \sum_{j=1}^K(K-j+1)(-\rho_j) - \sum_{j=1}^{i-1}((i-1)-j+1)(-\rho_j) \\
    & = \sum_{j=i}^K(K-j+1)(-\rho_j) + (K-i+1)\sum_{j=1}^{i-1}(-\rho_j),
\end{align*}
Now define the vector $\mathbf{v}_K \doteq \sum_{j=1}^K(K-j+1)\mathbf{e}_j$ and, for $i = 2, ..., K$, $\mathbf{u}_K^i \doteq (K-i+1)\sum_{j=1}^{i-1}\mathbf{e}_j + \sum_{j=i}^K(K-j+1)\mathbf{e}_j$, where $\mathbf{e}_j$ is the $j$-th standard basis vector on $\mathbb{R}^K$. It is easy to check that $\mathbf{v}_K'\mathscr{R} = \mathbf{e}_1'$ and  $(\mathbf{u}_K^i)'\mathscr{R} = \frac{1}{2}\mathbf{e}_i'$ for $i = 2,...,K$. Then, using the above relations,
\begin{align*}
    (\mathscr{R}^{-1}\rho)_1 & = \mathbf{e}_1'\mathscr{R}^{-1}\rho = (\mathbf{v}_K'\mathscr{R})\mathscr{R}^{-1}\rho = \sum_{j=1}^K(K-j+1)\rho_j = -\sum_{j=1}^Ka_j,
\end{align*}
and, for $i = 2, ..., K$,
\begin{align*}
    (\mathscr{R}^{-1}\rho)_i = \mathbf{e}_i'\mathscr{R}^{-1}\rho = (2(\mathbf{u}_K^i)'\mathscr{R})\mathscr{R}^{-1}\rho & = 2(\sum_{j=i}^K(K-j+1)\rho_j + (K-i+1)\sum_{j=1}^{i-1}\rho_j) \\
    & = -2\sum_{j=i}^Ka_j.
\end{align*}
From these identities, it is clear that $\mathscr{R}^{-1}\rho < 0$ if and only if the inequalities \eqref{stabcon} are satisfied. 
This proves the first statement in Theorem \ref{statdistthm}.

We will now show that, under the stability condition \eqref{stabcon}, the stationary measure has a product form density with respect to Lebesgue measure. Define $\Lambda \doteq \textnormal{diag}(AA^T)$. By \cite[Theorem 6.1]{harwil2} (see also \cite[Theorem 3.5]{srbmwilliams}), it suffices to show that $2AA' = \mathscr{R}\Lambda + \Lambda \mathscr{R}'$.

Observe that
\begin{align*}
    AA' &= 
      2
    \begin{pmatrix}
    1 & -1 & 0 & 0 & 0 &\cdots &\, &\,  0 \\
    -1 & 2 & -1 & 0 & 0 &\cdots &\, &\, 0 \\
    0 & -1 & 2 & -1 & 0 &\cdots &\, &\,  0 \\
    \vdots & \vdots & \vdots & \ddots & \vdots & \vdots & \vdots & \vdots\\
    0 & \cdots & \,& \,  &\, &\, -1& 2 &-1 \\
   0 & \cdots & \,& \,  &\, &\, 0& -1 &2 \\
    \end{pmatrix}.
\end{align*}
Hence, by the definition of $\Lambda$ and recalling the definition of $\mathscr{R}' $, we see that
\begin{align*}
    \Lambda \mathscr{R}' 
	    & = 2
	    \begin{pmatrix}
	    1 & -1 & 0 & \cdots & 0 \\
	    -1 & 2 & -1 & \cdots & 0 \\
	    0 & -1 & 2 & \cdots & 0 \\
	    \vdots & \vdots & \vdots & \ddots & \vdots \\
	    0 & 0 & 0 & \cdots & 2 \\
	    \end{pmatrix}
	    = AA',
	\end{align*}
Since $AA'$ is symmetric we now have that 
$$ \Lambda \mathscr{R}'  + \mathscr{R}\Lambda = 2 \Lambda \mathscr{R}'  = 2AA'.$$

To obtain an explicit expression for the stationary density, we appeal to \cite[Theorem 3.5]{srbmwilliams}, according to which the unique stationary distribution $\hat \pi$ has a density $\pi$ given by the formula
$$\pi(x) = c_{\pi}e^{\eta' x}, x \in \RR_+^K,$$
where $\eta = 2 \Lambda^{-1} \mathscr{R}^{-1}\rho$.
The formula \eqref{expstatdist} is now immediate on using the expressions for $\Lambda$ and $\mathscr{R}^{-1}\rho$.
\hfill
\qedsymbol

\textbf{Proof of Corollary \ref{cor:spcas}}

Fix $a \in (0,\infty), b \in \mathbb{R}$ and observe that $\sum_{j=1}^Ka_j = Ka - b$ and $\sum_{j = i}^Ka_j = (K-i+1)a$ for $i \in \{2,..., K\}$. Thus the stability condition \eqref{stabcon} holds if and only if $b<aK$.
Hence, from Theorem \ref{statdistthm},  $\bfZ(\cdot)$ in this parameter regime is positive recurrent if and only if $b < aK$, and the exponents of the density of the stationary distribution are $\eta_1 = -(aK-b), \eta_i = -(K+1-i)a$ for all $i \in \{2,...,K\}$, as was to be shown.
\qedsymbol

\section{Convergence of stationary distribution}\label{secconvstat}

In this section, we will prove Theorem \ref{statdistconv}. Throughout the section we assume that \eqref{stabcon} holds. Recall from the proof of Theorem \ref{statdistthm} that this inequality is equivalent to the inequality $\clr^{-1}\rho < 0$.

Our approach is inspired by the Lyapunov function method used in the proof of \cite[Theorem 3.1]{BudhirajLee}. In the following, let $G^n \doteq \frac{1}{\sqrt{n}} \NN_0^K$. Note that the processes $\hat{\Zbf}^n(\cdot)$ take values in $G^n$.

The following estimate will be central in constructing suitable  Lyapunov functions.

\begin{lemma}\label{lyaplem1} There are $t_0, D_3 \in (0, \infty)$ such that for every $t \in [t_0, \infty)$, $n \in \NN$ and $\bfz \in G^n$,
	\begin{equation}\label{lyaplim}\mathbb{E}_{\mathbf{z}}|\hat{\Zbf}^n(|\mathbf{z}|t)|^2 \le D_3\left( t|\bfz| + \frac{1 + t^2(|\bfz|^2+1)}{n}\right).\end{equation}
\end{lemma}

\begin{proof}
For $n \in \NN$, $\bfz \in G^n$ and $t \ge 0$, let
$$\hat \Vbf(t) \doteq \Gamma_2(\bfz+ \boldsymbol{\rho} \id)(t),$$
where $\Gamma = (\Gamma_1, \Gamma_2)$ is the Skorokhod map associated with the reflection matrix $\clr$.
From \eqref{eq:838}, under the measure where $\hat \bfZ^n(0)=\bfz$,
\begin{align}
	|\hat \bfZ^n(t) - \hat \Vbf(t)| 
	&= |\Gamma_2(\bfz + \boldsymbol{\rho} \id + A\tilde{\mathbf{M}}^n+ \tilde{\mathcal{E}}^n)(t)
	- \Gamma_2(\bfz+ \boldsymbol{\rho} \id)(t)|\nonumber\\
	&\le c_{\Gamma} \sup_{0\le s \le t}(|A\tilde{\mathbf{M}}^n(s)|+ |\tilde{\mathcal{E}}^n(s)|).\label{eq:516n}
\end{align}
Also, recalling the definitions of $\cle^n_i$ and $\tilde{\mathcal{E}}_i$, we see that
\begin{align*}
	\sup_{0\le s \le t}|\tilde{\mathcal{E}}^n(s)|
	\le 2 |a|_*\sqrt{K} \sum_{i,j \in [K], i \neq j} \int_0^t \mathbf{1}_{\{X^n_i(s) = X^n_j(s)\}} ds
\end{align*}
and thus using Corollary \ref{pairwisecorr} we have that
\begin{align}
	\EE_{\bfz} \sup_{0\le s \le t}|\tilde{\mathcal{E}}^n(s)|^2 \le
	4|a|^2_*K^3\frac{D_2}{n}(1+|\bfz|^2+ t^2). \label{eq:518n}
\end{align}
Also, using standard martingale inequalities, for some $\tilde D \in (0, \infty)$ (independent of $n, \bfz, t$),
\begin{equation}\label{eq:519n}\EE_{\bfz} \sup_{0\le s \le t}|A\tilde{\mathbf{M}}^n(s)|^2 \le \tilde D t.\end{equation}
Let
\begin{align*}
    \mathcal{S} \doteq \{v \in \mathbb{R}^{K}: \mathscr{R}^{-1}v < 0\}.
\end{align*}
From our  assumption $\clr^{-1}\rho<0$, we can find a $\delta_0>0$ such that 
    \begin{align*}\rho \in \mathcal{C}(\delta_0) \doteq \{ v \in \mathcal{S}: \inf_{u \in \partial\mathcal{S}}|u - v| > \delta_0\}.
\end{align*}
From \cite[Lemma 3.1 and its proof]{AtarBudDup} it follows that for all 
$t \ge (1+\frac{4c_{\Gamma}^2}{\delta_0})|\bfz|$, we have
$\hat \Vbf(t) = 0$.
Thus, for all $t \ge t_0 \doteq (1+\frac{4c_{\Gamma}^2}{\delta_0})$,
\begin{align*}
\mathbb{E}_{\mathbf{z}}|\hat{\Zbf}^n(t|\mathbf{z}|)|^2 & = \mathbb{E}_{\mathbf{z}}|\hat{\Zbf}^n(t|\mathbf{z}|) - \hat \Vbf(t|\bfz|)|^2\\
&\le 2c_{\Gamma}^2 (\mathbb{E}_{\mathbf{z}}\sup_{0\le s \le t|\bfz|}|A\tilde{\mathbf{M}}^n(s)|^2
+ \mathbb{E}_{\mathbf{z}}\sup_{0\le s \le t|\bfz|}|\tilde{\mathcal{E}}^n(s)|^2),
\end{align*}
where the last line follows from \eqref{eq:516n}. 

The result now follows on using the estimates from \eqref{eq:518n} and \eqref{eq:519n} in the above display.
\end{proof}

 The remainder of the proof of Theorem \ref{statdistconv} follows \cite{BudhirajLee}, with appropriate modifications. We describe these modifications below. 
 
 For $\delta \in (0,\infty)$ and a set $C \in \mathcal{B}(\mathbb{R}^K_+)$, define $\tau_C(\delta) \doteq \inf\{ t \geq \delta: \Zbf^n(t) \in C\}$ (suppressing $n$ for notational convenience).

\begin{lemma}\label{lyaplem2}
There exists $n_0 \in \mathbb{N}$, a compact set $C \subset \mathbb{R}_+^K$ and constants $\bar{\delta}, c \in (0,\infty)$ such that, for all $n \ge n_0$ and $\mathbf{z}\in G^n$,
\begin{equation}
    \mathbb{E}_{\mathbf{z}}\left(\int_0^{\tau_C(\bar{\delta)}}(1+|\hat{\mathbf{Z}}^n(s)|)ds\right) \leq c(1 + |\mathbf{z}|^2).
\end{equation}
\end{lemma}

\begin{proof}
Let $t_0$ be as in Lemma \ref{lyaplem1}. By Lemma \ref{lyaplem1}, for all $n \in \NN$ and $\bfz \in G^n$,
$$\frac{1}{|\bfz|^2}\mathbb{E}_{\mathbf{z}}|\hat{\Zbf}^n(|\mathbf{z}|t_0)|^2
\le D_3\left( \frac{t_0}{|\bfz|} +\frac{1}{n|\bfz|^2} + \frac{t_0^2(1+1/|\bfz|^2)}{n}\right).$$
Choose $L \in (0, \infty)$ such that
$$D_3\left( \frac{t_0}{L} +\frac{1}{L^2}\right) \le \frac{1}{4}$$
and then choose $n_0\in \NN$ such that
$$D_3\left(\frac{t_0^2(1+1/L^2)}{n_0}\right) \le \frac{1}{4}.$$
Letting $C \doteq \{z \in \RR_+^K: |z| \le L\}$, we have, for all $n \ge n_0$ and $\mathbf{z} \in C^c \cap G^n$,
\begin{equation}\label{il1}
\mathbb{E}_{\mathbf{z}}|\hat{\Zbf}^n(|\mathbf{z}|t_0)|^2 \le \frac{1}{2}|\bfz|^2.
\end{equation}
Using the representation \eqref{eq:838} for the dynamics of $\hat{\Zbf}^n(\cdot)$, the Lipschitz property of the Skorokhod map and the estimates \eqref{eq:518n} and \eqref{eq:519n}, it can be easily verified that there is a $c_0 \in (0, \infty)$, such that for all $n \in \NN$ and $\bfz \in G^n$, with $\sigma := t_0(|\bfz|\vee L)$,
\begin{align}\label{il2}
	\EE_{\bfz}  \int_0^{\sigma} (1+ |\hat{\Zbf}^n(t)|) dt \le c_0(1+|\bfz|^2).
\end{align}
Now we can follow exactly the arguments in \cite[Theorem 3.4]{BudhirajLee} to decompose the path of $\hat{\Zbf}^n(\cdot)$ in $[0,\tau_C(\delta)]$ into excursions. The bound in \eqref{il2} can be used to show that for large enough $n$, the expectation of the integral of the process $1+ |\hat{\Zbf}^n(\cdot)|$ over any such excursion can be bounded in terms of the expected initial value of $\hat{\Zbf}^n(\cdot)$ at the start of the excursion. By \eqref{il1}, this expected value contracts over successive excursions. These observations, along with Theorem 14.2.2 of \cite{meyn2012markov}, imply that, with $\bar{\delta} := t_0L$,
for all $n \ge n_0$ and $\mathbf{z}\in G^n$,
\begin{equation}
    \mathbb{E}_{\mathbf{z}}\left(\int_0^{\tau_C(\bar{\delta)}}(1+|\hat{\mathbf{Z}}^n(s)|)ds\right) \leq 3c_0(|\mathbf{z}|^2 + \tilde b)
\end{equation}
where $\tilde b \in (0, \infty)$ is a constant depending only on $D_3, L$ and $t_0$.
The result follows.
%
%
\end{proof}

The above estimate is sufficient to imply the existence of a stationary distribution for the pre-limiting process $\hat{\mathbf{Z}}^n(\cdot)$ for large enough $n$.
\begin{lemma}\label{statdistlem}
Suppose that the condition \eqref{stabcon} is satisfied. Recall $n_0$ from Lemma \ref{lyaplem2}. Then for each $n \ge n_0$, there is a unique stationary distribution $\hat{\pi}_n$ for the Markov process $\hat{\mathbf{Z}}^n(\cdot)$.
\end{lemma}
\begin{proof}
By Lemma \ref{lyaplem2}, $\EE_{\bfz}(\tau_C(\bar{\delta)}) \le c_0(1+|\bfz|^2)$ for all $\bfz \in G^n$ and $n \ge n_0$. This implies the positive recurrence, and hence the existence of a stationary distribution, of $\hat{\mathbf{Z}}^n(\cdot)$ for $n \ge n_0$. The uniqueness follows from the irreducibility of the process.
\end{proof}

We can now complete the proof of Theorem \ref{statdistconv}.\\

\noindent\textbf{Proof of Theorem \ref{statdistconv}:}
The existence assertion in the theorem is proved in Lemma \ref{statdistlem}. Now, we prove the convergence of stationary distributions.
Let $\bar \delta, C, c, n_0$ be as in the statement of Lemma \ref{lyaplem2}. 
Let, for $n \ge n_0$ and $\bfz \in G^n$,
$$V_n(\bfz) \doteq \EE_{\bfz} \int_0^{{\tau_C(\bar{\delta)}}} (1+ |\hat{\mathbf{Z}}^n(s)|) ds.$$
Then exactly 
as in the proof of \cite[Theorem 3.5]{BudhirajLee} (which in turn closely follows the proof of 
\cite[Proposition 5.4]{daimey}) we have that, there is a $\bar \kappa \in (0, \infty)$ such that for all $n \ge n_0$, $\bfz \in G^n$, and $t>0$,
\begin{equation}\label{eq:lyabnd}
    \frac{1}{t}\mathbb{E}_{\mathbf{z}}V_n(\hat{\mathbf{Z}}^n(t)) + \frac{1}{t}\int_0^t\left(1+\mathbb{E}_{\mathbf{z}}|\hat{\mathbf{Z}}^n(s)|\right)ds \leq \frac{1}{t}V_n(\mathbf{z}) + \bar\kappa.
\end{equation}
Rest of the proof is exactly as the proof of \cite[Theorem 3.2]{BudhirajLee}, however we provide details for the sake of completeness.

It is sufficient to show that $\int_{G^n}(1+|\mathbf{z}|)d\hat{\pi}^n(\mathbf{z}) \leq \bar{\kappa}$ for every $n \geq n_0$. From this, the asserted convergence of stationary distributions follows by a standard subsequence argument and Theorem \ref{convtodiff} (see discussion after \cite[Theorem 3.2]{BudhirajLee}). Fix $t>0$. Let $m \in \mathbb{N}$, and for $\mathbf{z} \in G^n$,  set $V^m_n(\mathbf{z}) \doteq V_n(\mathbf{z})\wedge m$ and
\begin{align*}
    \Psi^m_n(\mathbf{z}) \doteq \frac{1}{t}V_n^m(\mathbf{z}) - \frac{1}{t}\mathbb{E}_{\mathbf{z}}V_n^m(\hat\Zbf^n(t)),\,\,\,\,\,\,\,\,\,\, \Psi_n(\mathbf{z}) \doteq \frac{1}{t}V_n(\mathbf{z}) - \frac{1}{t}\mathbb{E}_{\mathbf{z}}V_n(\hat\Zbf^n(t)).
\end{align*}
Note that $\Psi^m_n(\mathbf{z}) \nearrow \Psi_n(\mathbf{z})$ as $m \rightarrow \infty$ for all $\mathbf{z} \in G^n$ by the monotone convergence theorem. We show that $\Psi_n^m(\mathbf{z})$ is uniformly bounded from below. Observe that by \eqref{eq:lyabnd}, when $V_n(\mathbf{z}) \leq m$,
\begin{align*}
    \Psi_n^m(\mathbf{z}) =  \frac{1}{t}V_n^m(\mathbf{z}) - \frac{1}{t}\mathbb{E}_{\mathbf{z}}V_n^m(\hat\Zbf^n(t)) \geq \frac{1}{t}V_n(\mathbf{z}) - \frac{1}{t}\mathbb{E}_{\mathbf{z}}V_n(\hat\Zbf^n(t)) \geq -\bar{\kappa}.
\end{align*}
On the other hand, when $V_n(\mathbf{z}) \geq m$,
\begin{align*}
    \Psi_n^m(\mathbf{z}) =  \frac{1}{t}m - \frac{1}{t}\mathbb{E}_{\mathbf{z}}V_n^m(\hat\Zbf^n(t)) \geq \frac{1}{t}m - \frac{1}{t}m = 0.
\end{align*}
Hence, $\Psi_n^m(\mathbf{z}) \geq -\bar{\kappa}$, for all $\mathbf{z} \in G^n$. Applying Fatou's Lemma, we obtain
\begin{align}
\begin{split}\label{psiineq}
    \int_{G^n} \Psi_n(\mathbf{z}) d\hat{\pi}^n(\mathbf{z}) &\leq \liminf_{m \rightarrow \infty}\int_{G^n} \Psi_n^m(\mathbf{z}) d\hat{\pi}^n(\mathbf{z}) \\
    & = \liminf_{m \rightarrow \infty}(\int_{G^n} \frac{1}{t}V_n^m(\mathbf{z})  d\hat{\pi}^n(\mathbf{z}) - \frac{1}{t}\int_{G^n} \mathbb{E}_{\mathbf{z}}V_n^m(\mathbf{z})  d\hat{\pi}^n(\mathbf{z})) = 0,
\end{split}
\end{align}
where the last equality follows by the fact that $\hat{\pi}^n$ is a stationary measure of $\hat{\mathbf{Z}}^n(\cdot)$. Rearranging \eqref{eq:lyabnd}, we have
\begin{equation*}
    \Psi_n(\mathbf{z}) \geq  \frac{1}{t}\int_0^t\left(1+\mathbb{E}_{\mathbf{z}}|\hat{\mathbf{Z}}^n(s)|\right)ds - \bar{\kappa}.
\end{equation*}
Combining this and \eqref{psiineq}, we obtain
\begin{align*}
    0 \geq \int_{G^n} \Psi_n(\mathbf{z}) d\hat{\pi}^n(\mathbf{z}) & \geq \frac{1}{t}\int_0^t\int_{G^n}\left(1+\mathbb{E}_{\mathbf{z}}|\hat{\mathbf{Z}}^n(s)|\right)d\hat{\pi}^n(\mathbf{z})ds - \bar{\kappa} \\
    & = \int_{G^n}\left(1+|\mathbf{z}|\right)d\hat{\pi}^n(\mathbf{z}) - \bar{\kappa}.
\end{align*}
The theorem follows.
\qedsymbol

\section{Stability of $\hat{\mathbf{Z}}^n_c$ in the Unstable  Regime}\label{secunstab}

In this section, we prove Theorem \ref{unstabstatdist}, which considers the special case $a_1 = a-b$ and $a_i = a$ for $ i = 2, \ldots , K$ where $b> aK$ and $a\ge 0$. As noted in Corollary \ref{cor:spcas}, in this case, the diffusion $\bfZ$ is not positive recurrent and one can similarly check that $\hat \bfZ^n$ is not stable for all $n$ sufficiently large. However Theorem \ref{unstabstatdist} says that the process $\hat \bfZ^n_c$ introduced above Theorem \ref{unstabstatdist} does have nice stability properties.
Throughout the section we fix $K \ge 2$, $a\ge 0$ and $b> aK$ and let $a_1 = a-b$ and $a_i = a$ for $ i = 2, \ldots , K$.

We begin with the following elementary lemma.

\begin{lemma}\label{lem:skorprop}
	Let $\bfv \doteq (-a+b, -b, 0, \ldots , 0)' \in \RR^K$ and define $v(t) \doteq t\bfv$, $t \ge 0$. Let $\Gamma = (\Gamma_1, \Gamma_2)$ be the Skorokhod map associated with the reflection matrix $\clr$. Then $\Gamma_2(v)(t) = t\bfv_1$, where $\bfv_1 = (b/K-a, 0, \ldots ,0)'$.
\end{lemma}
\begin{proof}
	It suffices to check that $\bfv_0 \doteq \clr^{-1}(\bfv_1-\bfv)$ satisfies $\bfv_0(1) =0$ and $\bfv_0(i) \ge 0$ for $i = 2, \ldots K$.
Define $\mathbf{h} \doteq \bfv_1-\bfv = (- \frac{K-1}{K}b, b,0,...,0)'$. Recall the quantities $\mathbf{v}_K$ and $\mathbf{u}_K^i$ from the proof of Theorem \ref{statdistthm}, and the identities $\mathbf{v}_K'\mathscr{R} = \mathbf{e}_1'$ and $(\mathbf{u}_K^i)'\mathscr{R} = \frac{1}{2}\mathbf{e}_i'$ for $i = 2,...,K$. Observe that 
\begin{align*}
        (\mathscr{R}^{-1}\mathbf{h})_1 = \mathbf{e}_1'\mathscr{R}^{-1}\mathbf{h} = (\mathbf{v}_K'\mathscr{R})\mathscr{R}^{-1}\mathbf{h} &= \sum_{j=1}^K(K-j+1)h_j \\
        & = K(-\frac{K-1}{K}b) + (K-1)b = 0,
    \end{align*}
    and, for $i = 2, ..., K$,
    \begin{align*}
        (\mathscr{R}^{-1}\mathbf{h})_i = \mathbf{e}_i'\mathscr{R}^{-1}\mathbf{h} = 2((\mathbf{u}_K^i)'\mathscr{R})\mathscr{R}^{-1}\mathbf{h} &= 2\sum_{j=i}^K(K-j+1)h_j + 2(K-i+1)\sum_{j=1}^{i-1}h_j \\
        & = \frac{2(K+1-i)}{K}b \geq 0.
    \end{align*}
\end{proof}

Recall the processes $\tilde{\mathbf{M}}^n(\cdot)$ and $\tilde{\mathcal{E}}^n(\cdot)$ from the proof of Theorem \ref{convtodiff}.
\begin{lemma}\label{lem:martlln}
	There is a $D_4 \in (0, \infty)$ such that, for all $n \in \NN$
	$$\PP\left(\limsup_{t\to \infty}\frac{\sup_{0\le s \le t}(|\tilde{\mathbf{M}}^n(s)| + |\tilde{\mathcal{E}}^n(s)|)}{t} \le \frac{D_4}{\sqrt{n}}\right)  = 1.$$
\end{lemma}

\begin{proof}
First observe that by the Burkholder-Davis-Gundy inequality and the forms of the quadratic variations of $\tilde{M}^n_i(\cdot)$, there exists a $D \in (0,\infty)$ (independent of $n$) such that, for all $t \geq 0$,
\begin{align*}
    \mathbb{E}\sup_{0 \leq s \leq t}|\tilde{\mathbf{M}}^n(s)|^4 \leq Dt^2.
\end{align*}
Let $\epsilon > 0$. Then, for each $k \in \mathbb{N}$,
\begin{align*}
    \mathbb{P}\left(\sup_{0\le s \le t}|\tilde{\mathbf{M}}^n(s)| > \epsilon t\,\, \textnormal{for some } t \in [k,k+1]\right) & \leq \mathbb{P}(\sup_{0 \leq s \leq k+1}|\tilde{\mathbf{M}}^n(s)| > \epsilon k) \\
    & \leq \frac{1}{(\epsilon k)^4}\mathbb{E}\sup_{0 \leq s \leq k+1}|\tilde{\mathbf{M}}^n(s)|^4 \\
    & \leq \frac{D(k+1)^2}{(\epsilon k)^4},
\end{align*}
where the second to last step follows by Markov's inequality. Hence, the Borel-Cantelli lemma shows that
\begin{equation}\label{eq:limsupmgle}
    \mathbb{P}(\limsup_{t\rightarrow\infty} \sup_{0\le s \le t}|\tilde{\mathbf{M}}^n(s)|/t = 0) = 1.
\end{equation}
Recall the processes $M^n_i(\cdot)$ defined in the proof of Lemma \ref{qtightprop} and $\hat{Z}^n_{i,j}(\cdot)$, $M^n_{i,j}(\cdot)$, $\zeta^n_{i,j}(\cdot)$, $\mu^{i,j}(\cdot)$ defined in the proof of Lemma \ref{pairwiselemma}. Taking $n > 4|a|^2_*$, we have that for $i,j \in [K]$ with $i \neq j$,

\begin{align}
\begin{split}
    \sqrt{n}\int_0^t \mathbf{1}_{\{X^n_i(s) = X^n_j(s)\}} ds &\leq \sqrt{n}\sum_{k,l \in [K]: k \neq l}\int_0^t \mathbf{1}_{\{Q^n_k(s) = Q^n_l(s)\}} ds  \\
    & \le \sum_{k,l \in [K]: k \neq l}|\zeta^n_{k,l}(t)| \\
    & \leq \sum_{k,l \in [K]: k \neq l}(\sup_{0 \leq s \leq t}|M^n_{k,l}(s)| + \hat{Z}^n_{k,l}(0)+ \frac{1}{\sqrt{n}}\int_0^t|\mu^{k,l}|(s)) \\
    & \leq \sum_{k,l \in [K]: k \neq l}(\sup_{0 \leq s \leq t}|M^n_{k,l}(s)| + \hat{Z}^n_{k,l}(0)+\sqrt{n}\int_{0}^{t}\mathbf{1}_{\{Q_l(s) = 0\}}ds \\
    &\,\,\,\,\,\,\,\,\,\,\,\,\,\,\,\,\,\,\,\,\,\,\,\,\,\,\,\,\,\,\,\,\,\,\,\,\,\,\,\,\,\,\,+ \sqrt{n}\int_{0}^{t}\mathbf{1}_{\{Q_k(s) = 0\}}ds + 2|a|_*t).
\end{split}
\end{align}
where the second inequality follows by \eqref{qijltineq}, the third inequality follows by the Lipschitz property of the one-dimensional Skorokhod map on $\mathbb{R}_+$ and the fourth inequality follows by \eqref{eq:muij}.

Using the Lipschitz property of the one-dimensional Skorokhod map to the bound above as in the proof of Lemma \ref{qtightprop}, we obtain
\begin{equation}\label{xijineqbloc}
\sqrt{n}\int_0^t \mathbf{1}_{\{X^n_i(s) = X^n_j(s)\}} ds \le \mathscr{M}^n(t) + \hat{C}^n + 4K^2|a|_*t,
\end{equation}
where
\begin{align*}
    \mathscr{M}^n(t) &\doteq \sum_{k,l \in [K]: k \neq l}\sup_{0 \leq s \leq t}|M^n_{k,l}(t)| + 2K\sum_{k=1}^K\sup_{0 \leq s \leq t}|M_k^n(s)|,\\
    \hat{C}^n &\doteq \sum_{k,l \in [K]: k \neq l}\hat{Z}^n_{k,l}(0)+2K\sum_{k=1}^K|\hat{Q}^n_k(0)|.
\end{align*}
By a similar Borel-Cantelli argument as was given above, 
\begin{equation}
    \mathbb{P}(\limsup_{t\rightarrow\infty}\mathscr{M}^n(t)/t = 0) = 1.
\end{equation}
Hence, dividing the inequality \eqref{xijineqbloc}  on both sides by $\sqrt{n}t$ and taking the limsup in $t$, we conclude that, almost surely, for any $i,j \in [K]$ with $i \neq j$,
\begin{align}
\begin{split}\label{limsupint}
    \limsup_{t\rightarrow\infty}\frac{1}{t}\int_0^t \mathbf{1}_{\{X^n_i(s) = X^n_j(s)\}} ds \leq \frac{4K^2|a|_*}{\sqrt{n}}.
\end{split}
\end{align}
Combining this with inequality \eqref{errorbound}, 
\begin{align*}
    \limsup_{t \rightarrow \infty}\frac{1}{t} \sup_{0\le s \le t}|\tilde{\mathcal{E}}^n(s)| \leq \frac{4K^4|a|^2_*}{\sqrt{n}}.
\end{align*}
Thus, setting $D_4 \doteq 4K^4|a|^2_*$, the above and \eqref{eq:limsupmgle} combined yield the desired result.
\end{proof}

\begin{lemma}\label{lem:staypos}
	There is a $n_0 \in \NN$ such that for all $n\ge n_0$,
	$$\PP(\liminf_{t\to \infty} \hat Z^n_1(t)/t  \ge (b-Ka)/2K) = 1.$$
\end{lemma}
\begin{proof}
	Recall $v(\cdot)$ from Lemma \ref{lem:skorprop}. From \eqref{eq:838} and the Lipschitz property of the Skorokhod map $\Gamma$ with reflection matrix $\mathscr{R}$,
	\begin{align*}
		\frac{1}{t} |\hat \bfZ^n(t) - \Gamma_2(v)(t)| \le c_{\Gamma}\frac{|\hat \bfZ^n(0)| + |A| \sup_{0\le s \le t}|\tilde{\mathbf{M}}^n(s)| + \sup_{0\le s \le t}|\tilde{\mathcal{E}}^n(s)|}{t}.
	\end{align*}
Set $n_0 = (\frac{2Kc_\Gamma (|A|+1)D_4}{b-Ka})^2$. Taking $n \geq n_0$, the above yields almost surely
    \begin{align*}
		\liminf_{t \rightarrow \infty}\frac{1}{t} |\hat Z_1^n(t) - \frac{(b-Ka)t}{K}| \le \liminf_{t \rightarrow \infty}\frac{1}{t} |\hat \bfZ^n(t) - \Gamma_2(v)(t)| 
         \leq c_{\Gamma}\frac{D_4(|A|+1)}{\sqrt{n}}
         \leq \frac{b-Ka}{2K}
	\end{align*}
where the first inequality follows from Lemma \ref{lem:skorprop} and the second inequality follows from Lemma \ref{lem:martlln}. The result is immediate from this display.
\end{proof}
We now introduce a $(K-1)$-dimensional process $\hat \bfU^n$ which is analogous to $\bfZ^n_c$  but unlike the latter process it is a Markov process.
Let $\{\tilde A_i, \tilde D_i, 1 \le i \le K\}$ be a collection of $2K$ mutually independent rate $1$ Poisson processes.
Given $\bfx = (x_2, \ldots x_K)' \in G^n_{K-1} \doteq \frac{1}{\sqrt{n}}\NN_0^{K-1}$, let $\hat \bfU^n(t) \doteq (\hat U^n_{2}(t), \ldots \hat U^n_{K}(t))$ be a stochastic process with values in $G^n_{K-1}$, defined as follows. For $i = 2, \ldots , K$, 
\begin{align}
	\hat U^n_i(t) &\doteq x_i + \tilde A_i(n\int_0^t \tilde\alpha^n_i(s)\mathbf{1}_{\{\hat U^n_{i+1}(s)>0\}}ds) - \tilde D_i(n\int_0^t \tilde\delta^n_i(s)\mathbf{1}_{\{\hat U^n_{i}(s)>0\}}ds)\nonumber\\
	&\quad -\tilde A_{i-1}(n\int_0^t \tilde\alpha^n_{i-1}(s)\mathbf{1}_{\{\hat U^n_{i}(s)>0\}}ds) + \tilde D_{i-1}(n\int_0^t \tilde\delta^n_{i-1}(s)\mathbf{1}_{\{\hat U^n_{i-1}(s)>0\}}ds), \label{eq:722}
\end{align}
where, by convention $\hat U^n_{1}(t) \doteq 1$, and,
with $\hat V^n_j(t) \doteq \sum_{i=2}^j \hat U^n_i(t)$, $j = 2, \ldots , K$, and  $\hat V^n_1(t) \doteq 0$,
$$
\tilde \delta^n_i(t) \doteq \sum_{j=1}^K \mathbf{1}_{\{\hat V^n_j(t) =\hat V^n_i(t)\}}, \; \tilde \alpha_i^n(t) \doteq \sum_{j=1}^K (1 - \frac{a_j}{\sqrt{n}})\mathbf{1}_{\{\hat V^n_j(t) =\hat V^n_i(t)\}}, \; i \in [K], \; t\ge 0.$$
Intuitively, $\hat{\bfU}^n(\cdot)$ may be viewed as modeling the $K-1$ gaps between the ranked $K$ queues in the analogous system without reflection at zero, that is, if the queues were allowed to take negative values.

Proof of the next lemma follows along the lines of Theorem \ref{convtodiff} and Theorem  \ref{statdistconv}. 

\begin{lemma}\label{lem:unerg}
	There exists $n_0 \in \mathbb{N}$ such that for all $n \ge n_0$, the $G^n_{K-1}$ valued Markov process $\{\hat \bfU^n(\cdot)\}$ has a unique stationary distribution $\hat\pi^{n, \bfU}$. Furthermore, denoting for $\bfx \in G^n_{K-1}$, by $\PP_{\bfx}$ the probability measure under which  $\hat \bfU^n(0)=\bfx$, we have that, for each $\bfx, \bfy \in G^n_{K-1}$, $n \ge n_0$,
	\begin{equation}\label{erg}
	\PP_{\bfx}(\bfU^n(t) =\bfy) \to \hat\pi^{n, \bfU}(\bfy), \mbox{ as } t \to \infty.
	\end{equation}
	Finally, defining $\hat \pi^{n, \bfU} \in \clp(\RR_+^{K-1})$ as
	$$\hat \pi^{n, \bfU}(A) \doteq \sum_{\bfx \in A} \hat\pi^{n, \bfU}(\bfx),$$
	we have that $\hat \pi^{n, \bfU} \Rightarrow \hat \pi^{\bfU}$ where 
	$$\hat \pi^{\bfU} \sim \bigotimes_{i=1}^{K-1}\textnormal{Exp}((\frac{K-i}{K})b).$$
\end{lemma}

\begin{proof}[Proof Sketch]

Note that, analogous to \eqref{eq:838}, $\hat\bfU^n(t) = \bar{\Gamma}_2 \left(\hat\bfU^n(0) + \bar{\boldsymbol{\rho}}\id + \bar{A} \bar\bfM^n + \bar{\mathcal{E}}^n\right)(t)$, $t \ge 0$, where $\bar{\Gamma} = (\bar{\Gamma}_1,\bar{\Gamma}_2)$ is the Skorokhod map associated with the $(K-1) \times (K-1)$ reflection matrix $\bar{R}$ obtained by deleting the first row and first column of $R$, $\bar{\boldsymbol{\rho}} = (-b,0,\dots,0)'$, $\bar{A}$ is the $(K-1) \times K$ sub-matrix formed by deleting the first row of $A$, $\bar\bfM^n$ is the $K$-dimensional martingale process obtained from the compensated arrival and departure processes similarly as $\tilde{\bfM}^n$, and $\bar{\mathcal{E}}^n$ is the `error term' which is of a smaller order (in $n$) compared to the other terms. The same approach as in the proof of Theorem \ref{convtodiff} shows that $\hat{\Ubf}^n(\cdot) \Rightarrow\,\, \Ubf(\cdot)$ in $\mathcal{D}([0,\infty): \mathbb{R}_+^K)$, where $\bfU(t) = \bar{\Gamma}_2\left(\bfx + \bar{\boldsymbol{\rho}}\id + \bar{A}\bfB\right)(t), \, t \ge 0,$ where $\bfB$ is a $K$-dimensional standard Brownian motion. Now, observing that $\bar{R}^{-1}\bar{\boldsymbol{\rho}}<0$, and using results from \cite{harwil,harwil2} as before, $\bfU(\cdot)$ has a unique stationary distribution given by $\hat \pi^{\bfU}$ in the lemma.

Proceeding in the same way as in the proof of Theorem \ref{statdistconv}, we obtain the existence of the stationary distribution $\hat \pi^{n, \hat\bfU}$ of $\hat\bfU^n(\cdot)$ for large enough $n$ and the convergence of $\hat \pi^{n, \bfU}$ to $\hat \pi^{\bfU}$ as $n \rightarrow \infty$. The time convergence \eqref{erg} follows from the positive recurrence and irreducibility of $\hat\bfU^n(\cdot)$ \cite[Theorem (21), Section 6.9, p.261]{grimmettprobability}.
\end{proof}

For $n \in \NN$, $\bfx, \bfy \in G^n_{K-1}$ and $t \ge 0$, let 
$$\clt_t^n(\bfx, \bfy) \doteq \PP_{\bfx}(\hat \bfU^n(t)= \bfy).$$
Recall the processes $\{A_i, D_i, X^n_i, i \in [K]\}$ from \eqref{ordqs}.
Define for $m \in \NN$, $i \in [K]$ and $t\ge 0$,
\begin{align*}
	\tilde A^m_i(t) \doteq A_i\left(t+ n\int_0^m \alpha^n_i(s)\mathbf{1}_{\{X^n_i(s) < X^n_{i+1}(s)\}}ds\right) - A_i\left( n\int_0^m \alpha^n_i(s)\mathbf{1}_{\{X^n_i(s) < X^n_{i+1}(s)\}}ds\right),\\
	\tilde D^m_i(t)  \doteq D_i\left(t+ n\int_0^m \delta^n_i(s)\mathbf{1}_{\{X^n_{i}(s) > X^n_{i-1}(s)\}}ds\right) - D_i\left( n\int_0^m \delta^n_i(s)\mathbf{1}_{\{X^n_{i}(s) > X^n_{i-1}(s)\}}ds\right).
\end{align*}
Now define for $n,m \in \NN$, $\hat \bfU^{n,m}$ by replacing $(\tilde A_i, \tilde D_i)$ in \eqref{eq:722} with $(\tilde A_i^m, \tilde D_i^m)$ and $\bfx$ with
$\hat \bfZ^n_c(m)$, namely
for $i = 2, \ldots , K$, $t\ge 0$,
\begin{align}
	\hat U^{n,m}_i(t) &\doteq \hat \bfZ^n_i(m) + \tilde A^m_i\left(n\int_0^t \tilde\alpha^{n,m}_i(s)\mathbf{1}_{\{\hat U^{n,m}_{i+1}(s)>0\}}ds\right) - \tilde D^m_i\left(n\int_0^t \tilde\delta^{n,m}_i(s)\mathbf{1}_{\{\hat U^{n,m}_{i}(s)>0\}}ds\right)\nonumber\\
	&\quad -\tilde A^m_{i-1}\left(n\int_0^t \tilde\alpha^{n,m}_{i-1}(s)\mathbf{1}_{\{\hat U^{n,m}_{i}(s)>0\}}ds\right) + \tilde D^m_{i-1}\left(n\int_0^t \tilde\delta^{n,m}_{i-1}(s)\mathbf{1}_{\{\hat U^{n,m}_{i-1}(s)>0\}}ds\right), \label{eq:7222}
\end{align}
where, by convention $\hat U^{n,m}_{1}(s) \equiv 1$, and,
with $\hat V^{n,m}_j(t) \doteq \sum_{i=2}^j \hat U^{n,m}_i(t)$, $j = 2, \ldots , K$, and  $\hat V^{n,m}_1(t) \doteq 0$,
$$
\tilde \delta^{n,m}_i(t) \doteq \sum_{j=1}^K \mathbf{1}_{\{\hat V^{n,m}_j(t) =\hat V^{n,m}_i(t)\}}, \; \tilde \alpha_i^{n,m}(t) \doteq \sum_{j=1}^K (1-\frac{a_j}{\sqrt{n}})\mathbf{1}_{\{\hat V^{n,m}_j(t) =\hat V^{n,m}_i(t)\}}, \; i \in [K], \; t\ge 0.$$
In words, $\hat{\bfU}^{n,m}(\cdot)$ denotes the $(K-1)$-dimensional gap process of the ranked $K$ queues when the queues are non-negative (reflected at zero) till time $m$ and then are allowed to take negative values (no reflection at zero).
Let, for $n,m \in \NN$,
$$\clf^n_m \doteq \sigma\{ X^n_i(t), A_i\left(n\int_0^t \alpha^n_i(s)\mathbf{1}_{\{X^n_i(s) < X^n_{i+1}(s)\}}ds\right), D_i\left(n\int_0^t \delta^n_i(s)\mathbf{1}_{\{X^n_{i}(s) > X^n_{i-1}(s)\}}ds\right), t \le m, i \in [K]\}.$$
The following lemma holds.
\begin{lemma}
	For $n \in \NN$, $\bfy \in G^n_{K-1}$, and $m \in \NN$
	$$\PP(\hat \bfU^{n,m}(t) = \bfy \mid \clf^n_m) = \clt_t^n(\hat \bfZ^n_c(m), \bfy).$$
\end{lemma}

\begin{proof}
This is immediate on using the Markov property of $\hat\bfU^n$.
\end{proof}

We now complete the proof of Theorem \ref{unstabstatdist}.\\

\noindent
 \textbf{Proof of Theorem \ref{unstabstatdist}.}
Let for $\om \in \Om$,
$$t_0(\om) \doteq \inf\{t\ge 0: \hat Z^n_1(s, \om) \ge \frac{b-Ka}{4K}s \mbox{ for all } s \ge t\}.$$
From Lemma \ref{lem:staypos} we see that $t_0(\om)<\infty$ a.s.
Now note that, for $m \in \NN$ and $\bfz \in G^n_{K-1}$,
\begin{align*}
	|\PP(\hat \bfZ^n_c(t+m) = \bfz) -\PP(\hat \bfU^{n,m}(t) = \bfz)| \le  2 \PP(t_0 > m).
\end{align*}
Also, from Lemma \ref{lem:unerg},
\begin{align*}
	\lim_{t\to \infty} \PP(\hat \bfU^{n,m}(t) = \bfz) = \lim_{t\to \infty} \EE(\clt_t^n(\hat \bfZ^n(m), \bfz)) = \hat\pi^{n, \bfU}(\bfz).
\end{align*}
Thus
\begin{align*}
	\limsup_{t\to \infty}|\PP(\hat \bfZ^n_c(t) = \bfz) - \hat\pi^{n, \bfU}(\bfz)|  &= \limsup_{t\to \infty}|\PP(\hat \bfZ^n_c(t+m) = \bfz) - \hat\pi^{n, \bfU}(\bfz)|\\
	&\le  \limsup_{t\to \infty} |\PP(\hat \bfU^{n,m}(t) = \bfz) - \hat\pi^{n, \bfU}(\bfz)| + 2 \PP(t_0 > m)\\
	& = 2 \PP(t_0 > m).
\end{align*}
Since $t_0<\infty$, a.s., we have the first statement in Theorem \ref{unstabstatdist} on sending $m \to \infty$ in the above display.
The second statement is immediate from Lemma \ref{lem:unerg}. \hfill \qed
$\,$\\

 \noindent {\bf Acknowledgements}
Research supported in part by  the RTG award (DMS-2134107) from the NSF.  SB was supported in part by the NSF-CAREER award (DMS-2141621).
AB was supported in part by the NSF (DMS-2152577). 

{\cb We acknowledge the valuable feedback from two referees and an associate editor that significantly improved this article.}

\bibliographystyle{plain}
\bibliography{queue_ref}

\vspace{\baselineskip}

\noindent{\scriptsize {\textsc{\noindent S. Banerjee, A. Budhiraja, and B. Estevez\newline
Department of Statistics and Operations Research\newline
University of North Carolina\newline
Chapel Hill, NC 27599, USA\newline
email: sayan@email.unc.edu
\newline
email: budhiraj@email.unc.edu
\newline
email: bestevez@live.unc.edu
 \vspace{\baselineskip} } }}

\end{document}